\documentclass[10pt,a4paper,reqno]{article}
\usepackage{graphicx} %,bbold,bbm,mathbbol,
\usepackage[english]{babel} %for biblatex
\usepackage{csquotes} %for biblatex
\usepackage{filecontents} %for biblatex
\usepackage[dvipsnames]{xcolor}
\usepackage{color}
\usepackage[colorinlistoftodos]{todonotes}
\usepackage{tensor}
\usepackage{thmtools}
\usepackage{microtype}

\usepackage{fourier}

\usepackage[
backend=biber,
hyperref=true,
backref=true,
isbn=false,
doi=true,
url=false,
natbib=true,
eprint=true,
useprefix=true,
maxcitenames=99,
maxbibnames=99,  
maxalphanames=99, 
minalphanames=99,
safeinputenc,
style=alphabetic,
citestyle=alphabetic,
block=space,
datamodel=ext-eprint,
]{biblatex}
\usepackage[
hypertexnames = false,
colorlinks    = true,
citecolor     = teal,
linkcolor     = MidnightBlue,
urlcolor      = blue,
linktocpage = false,
breaklinks
]{hyperref}
\renewbibmacro{in:}{} %%% for removing In : in biblatex %%%%%

\DeclareSourcemap{
	\maps[datatype=bibtex]{
		\map{
			\step[fieldsource=pmid, fieldtarget=pubmed]
		}
	}
}

\makeatletter
\DeclareFieldFormat{arxiv}{%
	arXiv\addcolon\space
	\ifhyperref
	{\href{http://arxiv.org/\abx@arxivpath/#1}{%
			\nolinkurl{#1}%
			\iffieldundef{arxivclass}
			{}
			{\addspace\texttt{\mkbibbrackets{\thefield{arxivclass}}}}}}
	{\nolinkurl{#1}
		\iffieldundef{arxivclass}
		{}
		{\addspace\texttt{\mkbibbrackets{\thefield{arxivclass}}}}}}
\makeatother
\DeclareFieldFormat{pmcid}{%
	PMCID\addcolon\space
	\ifhyperref
	{\href{http://www.ncbi.nlm.nih.gov/pmc/articles/#1}{\nolinkurl{#1}}}
	{\nolinkurl{#1}}}
\DeclareFieldFormat{mrnumber}{%
	MR\addcolon\space
	\ifhyperref
	{\href{http://www.ams.org/mathscinet-getitem?mr=MR#1}{\nolinkurl{#1}}}
	{\nolinkurl{#1}}}
\DeclareFieldFormat{zbl}{%
	Zbl\addcolon\space
	\ifhyperref
	{\href{http://zbmath.org/?q=an:#1}{\nolinkurl{#1}}}
	{\nolinkurl{#1}}}
\DeclareFieldAlias{jstor}{eprint:jstor}
\DeclareFieldAlias{hdl}{eprint:hdl}
\DeclareFieldAlias{pubmed}{eprint:pubmed}
\DeclareFieldAlias{googlebooks}{eprint:googlebooks}

\renewbibmacro*{eprint}{%
	\printfield{arxiv}%
	\newunit\newblock
	\printfield{jstor}%
	\newunit\newblock
	\printfield{mrnumber}%
	\newunit\newblock
	\printfield{zbl}%
	\newunit\newblock
	\printfield{hdl}%
	\newunit\newblock
	\printfield{pubmed}%
	\newunit\newblock
	\printfield{pmcid}%
	\newunit\newblock
	\printfield{googlebooks}%
	\newunit\newblock
	\iffieldundef{eprinttype}
	{\printfield{eprint}}
	{\printfield[eprint:\strfield{eprinttype}]{eprint}}}

\usepackage[margin=2.1cm,includehead]{geometry}
\usepackage{seqsplit}
\usepackage{xstring}

\usepackage{amsmath,stmaryrd, upgreek}
\usepackage{amsthm,amssymb}
\usepackage{latexsym}
\usepackage{amscd}
\usepackage{mathrsfs}
\usepackage{url}
\usepackage{mathtools}
\usepackage[nameinlink]{cleveref}   %%%%nameinlink makes Theorem also clickable
\usepackage{doi}

\usepackage[T1]{fontenc}
\usepackage{tikz-cd}
\usepackage{titlesec}

\usepackage[titletoc,toc,title]{appendix}

\makeatletter %gets rid of Contents in toc
\renewcommand\tableofcontents{%
	\@starttoc{toc}%
	
}
\makeatother

\usepackage{enumerate}
\usepackage{slashed}
\graphicspath{}
\usepackage{fancyhdr} %gives heading and author name on alternate pages

\usepackage{changepage}

\usepackage{cancel}
\usepackage[all]{xy}
\usepackage{multicol}

\usepackage{charter}

\newcounter{noteCounter}
\newcommand\shorttitle{Ricci flow with torsion on $\chi(M)>0$} %title which appear on alternate pages
\newcommand\authors{S. Dwivedi} %author name appears on alternate pages

\newcounter{commentCounter}
\titleformat{\section}    
{\normalfont\large\bfseries\center}{\thesection.}{1em}{}
\makeatletter
\newcommand*{\rom}[1]{\expandafter\@slowromancap\romannumeral #1@}
\makeatother
\newtheorem{theorem}{Theorem}[section]
\newtheorem{corollary}[theorem]{Corollary}
\newtheorem{lemma}[theorem]{Lemma}
\newtheorem{proposition}[theorem]{Proposition}

\theoremstyle{definition}
\newtheorem{definition}[theorem]{Definition}
\newtheorem{remark}[theorem]{Remark}
\newtheorem*{ack}{Acknowledgments}
\numberwithin{equation}{section}
\def\bR{\mathbb R}

\def\Sph{\mathbb S^2}
\def\bRP{\mathbb {RP}}

\def\pt{\partial}
\def\del{\nabla}
\def\G2{\mathrm{G}_2}
\def\g2{\varphi}
\def\S7{\mathrm{Spin}(7)}
\def\s7{\Phi}

\def\cE{\mathcal{E}}
\def\cF{\mathcal{F}}

\def\cL{\mathcal{L}}

\def\cN{\mathcal{N}}

\def\Spin7{\mathrm{Spin(7)}}

\def\dots7{\Dot{\Phi}}

\DeclareMathOperator\diver{div}

\DeclareMathOperator\grad{grad}

\DeclareMathOperator\Vol{Vol}

\newcommand{\Conf}{\operatorname{Conf}}

\newcommand\xqed[1]{%
	\leavevmode\unskip\penalty9999 \hbox{}\nobreak\hfill
	\quad\hbox{#1}}
\newcommand\demo{\xqed{$\blacktriangle$}}

\begin{document}

	\title{Ricci flow with metric torsion on surfaces of positive Euler characteristic}
	\author{Shubham Dwivedi}
	\date{}
	
	\maketitle

\begin{abstract}
We study an adapted Ricci flow of connections with metric torsion on surfaces with positive Euler characteristic. We first prove that there do not exist any nontrivial solitons of the flow on the $2$-sphere thus confirming a conjecture of Branding--Kr\"oncke \cite{BrandingKroencke}. We give an explicit family of torsion data for which the corresponding global solutions fail to converge on $\mathbb{S}^2$. Nevertheless, we provide several sufficient conditions for the convergence of the flow to a stationary point. We first prove that the normalized adapted Ricci flow always converges on $\mathbb{RP}^2$, which completely answers a question in \cite{BrandingKroencke}. Using this, we deduce that the flow converges on $\mathbb{S}^2$ whenever the initial metric and the torsion one-form are antipodally symmetric. We also prove a {\L}ojasiewicz--Simon gradient inequality for the flow and use it to prove convergence to a stationary point provided the solution is close to an arbitrary stationary point. 
\end{abstract}
	
	\begin{adjustwidth}{0.95cm}{0.95cm}
		\tableofcontents
	\end{adjustwidth}
	%\listoftodos
	% %\newpage
	
	\let\thefootnote\relax

	\section{Introduction}\label{sec:intro}
The main purpose of this paper is to study properties of the Ricci flow of metric connections with torsion on two dimensional manifolds with positive Euler characterisitic. The Ricci flow of metric connections with torsion on a Riemann surface was studied by Branding and Kröncke in \cite{BrandingKroencke} where they studied properties of the flow and used it to prove an analogue of the uniformization theorem for metric connections with torsion when the surface has non-positive Euler characteristic. Let $M$ be a Riemannian surface with a Riemannian metric $g$. If $\del^{LC}$ denotes the Levi-Civita connection of $g$ then any affine connection $\del^T$ on $M$ can be written as 
	\begin{align*}
		\nabla^{ T}_XY=\nabla^{LC}_XY+A(X,Y), \qquad X, Y\in \Gamma(TM),
	\end{align*}
where $A$ is a $(2,1)$-tensor field. If in addition, we want the connection $\del^T$ be a \emph{metric connection}, that is, for all vector fields $X,Y,Z$ we have
	\begin{align*}
		X(g(Y,Z))=g(\nabla^T_XY,Z)+g(Y,\nabla^T_XZ),
\end{align*}
then $A$ must be skew-adjoint.

Since $M$ is $2$-dimensional, there exists a vector field $V$ such that the tensor $A$ takes the form
	\begin{align*}
		A(X,Y)=g(X,Y)V-g(V,Y)X,
	\end{align*}
and the torsion tensor $T$ of the connection $\del^T$ is then given by 
	\begin{align*}
		%\label{torsion-t}
	T(X,Y)=A(X,Y)-A(Y,X)=g(V,X)Y-g(V,Y)X.
	\end{align*}
Throughout the article, we will denote the vectorial torsion with respect to $g$ by
\begin{align*}
V_g=\alpha^{\sharp_g},	
\end{align*}
for some fixed $\alpha\in \Omega^1(M)$. The curvatures of the connection $\del^T$ have contributions from the tensor $T$ as well. In particular, the scalar curvature $R^T_g$ is
\begin{align}\label{eq:scalar-torsion}
R^T_g = R_g+2\diver_gV_g = R_g+2\diver_g\alpha,
\end{align}
where, by abuse of notation, $\diver_g\alpha$ denotes the divergence of the metric-dual vector field $\alpha^{\sharp_g}$. Here $R_g$ is the scalar curvature of the Levi-Civita connection. For a proof of \cref{eq:scalar-torsion} and more on geometry of metric connections with vectorial torsion, we refer the reader to \cite{agricola-kraus}. We note that if $e^ug$ is a metric conformal to $g$ then the vector field $e^{-u}V$ induces a connetion with the same vectorial torsion.

\medskip

Motivated by the problem of finding a metric conformal to $(M, g)$ with a given vectorial torsion $T$ or equivalently $\alpha$ such that $R^T_g$ is constant, Branding and Kr\"oncke studied the Ricci flow with metric torsion which we now describe. Let $(M^2, g_0)$ be a closed Riemann surface, $\alpha\in\Omega^1(M)$ be a fixed smooth one-form. The normalized adapted Ricci flow is the following evolution equation for a family of metrics $g(t)$ on $M$ all with fixed vectorial torsion $\alpha$,
\begin{align}\label{eq:rfeqn}
\frac{\partial g}{\partial t} =\left(r-R^T_g \right)g,\qquad g(0)=g_0,
\end{align}
where
\begin{align}\label{eq:totalscalar}
r = \frac{\displaystyle\int_{M}R^T_g d\mu_g}{\displaystyle\Vol(M,g)} = \frac{\displaystyle 4\pi \chi(M)}{\displaystyle \Vol(M,g)},
\end{align}
is the average scalar curvature with torsion. The second equality in \cref{eq:totalscalar} follows from Gauss--Bonnet and the divergence theorem:
\begin{align*}
\int_{M}R^T_g d\mu_g = \int_{M}R_g\,d\mu_g + 2\int_{M}\diver_g\alpha\,d\mu_g = 4\pi \chi(M).
\end{align*}
Along the flow \cref{eq:rfeqn}, the volume and hence $r$ are constant. If we write $g=e^u g_0$ for some $u\in C^{\infty}(M)$ then because the dimension is two, we have
\begin{align}\label{eq:scalarconformal}
R_g=e^{-u}\left(R_{g_0}-\Delta_{g_0}u\right) \qquad \text{and} \qquad \diver_g\alpha=e^{-u}\diver_{g_0}\alpha,
\end{align}
where for us $\Delta = \diver (\grad)$. Consequently,
\begin{align}\label{eq:scalartorsionconf}
R^T_g= e^{-u}\left(R_{g_0}-\Delta_{g_0}u +2\diver_{g_0}\alpha \right).
\end{align}
A similar flow in higher dimensions has been studied by Streets \cite{streets}. The flow is also  related to the  renormalization group flow in quantum field theory. 

\medskip

Inspired by the proof of the uniformization theorem on surfaces with a Riemannian metric using the Ricci flow in \cite{Hamilton}, \cite{chow-2sphere}, \cite{chen-lu-tian}, Branding and Kr\"oncke proved the following long-time existence and convergence result for the flow \cref{eq:rfeqn}.

\begin{theorem}[\cite{BrandingKroencke}, Thm. 1.1] Let $(M,g_0)$ be a closed Riemannian surface. Then there exists a unique solution $g(t)$ of \cref{eq:rfeqn} for all $t\in [0,\infty)$. If $\chi(M)\leq 0$ the solution converges to a metric of constant curvature with torsion as $t\to\infty$ and hence in this case, the conformal class of $g_0$ contains a (up to rescaling) unique metric $\bar{g}$ whose scalar curvature with respect to the torsion $\alpha$ is constant.
\end{theorem}

As mentioned in \cite{BrandingKroencke}, the case of the convergence of the flow on surfaces with $\chi(M) >0$ is a subtle one. In fact, this problem is related to the so-called \emph{prescribed curvature problem} and on $\Sph$ it is known as \emph{Nirenberg's problem}. There have been some outstanding papers on this problem using variational techniques \cite{ChangYang} as well as parabolic methods \cite{bfr}, \cite{Struwe}. One of the important results used in the proof of the uniformization theorem for $\chi(M)>0$ using the ordinary Ricci flow is that any gradient Ricci soliton on $\Sph$ has constant scalar curvature, a result first proved by Hamilton \cite{Hamilton} and also by Chen--Lu--Tian \cite{chen-lu-tian}, see also \cite[Chapter 4]{brendle}. We first define \emph{self-similar solutions} or \emph{solitons} of the adapted Ricci flow following \cite{BrandingKroencke}.

\begin{definition}\label{def:soliton}
	A self-similar solution or a soliton of the normalized adapted Ricci flow \cref{eq:rfeqn} on a Riemann surface $M$ is a triple $(g,\alpha, X)$, where $X$ is a smooth vector field satisfying
	\begin{align}\label{eq:soleqn}
		\cL_Xg=(r-R^T_g)g \qquad \text{and}\qquad  \cL_X\alpha=0.
	\end{align}
	Ths soliton is trivial if
	\begin{align*}
		R^T_g\equiv r \qquad \text{and}\qquad \cL_Xg=0,
	\end{align*}
and it is a gradient soliton if $X=\del f$ for some smooth function $f$.	
\end{definition}	
\noindent
We note that the first identity in \cref{eq:soleqn} implies that $X$ is a conformal vector field on $M$.

\medskip
	
Solitons for \cref{eq:rfeqn} on $\Sph$ were studied in \cite{BrandingKroencke} where the authors showed that the round sphere cannot support a nontrivial adapted soliton and moreover, any non-trivial soliton on $\Sph$ can neither be gradient soliton nor can have divergence-free $X$. The authors conjectured that any soliton on $\Sph$ must be trivial. Our first main result confirms this conjecture.

\begin{theorem}\label{thm:nosolitononS2}
Every smooth self-similar solution of the normalized adapted Ricci flow \cref{eq:rfeqn}
on $\Sph$ is trivial.	
\end{theorem}
As mentioned above, the soliton vector field $X$ is a conformal vector field on $\Sph$. Since all metrics on $\Sph$ determine the standard conformal structure, the flow generated by $X$ is a one-parameter subgroup of $\Conf^+(\Sph)\cong\operatorname{PSL}(2,\mathbb C)$. Such flows, and equivalently vector fields $X$, are classified as \textbf{elliptic, parabolic} and \textbf{hyperbolic.} We show that in the parabolic and the hyperbolic case, the condition $\cL_X\alpha=0$ forces the exact part of $\alpha$ to vanish, and hence $\diver_g\alpha=0$. As a result, the adapted soliton is an ordinary Ricci soliton on $\Sph$ which then is trivial as a result of Kazdan--Warner identity \cite[Prop. 5.21]{ChowKnopf}. For treating the elliptic case, we look at a Liouville energy \cref{eq:energy} and prove that it is monotonically decreasing along the flow \cref{eq:rfeqn}. This is proved in \Cref{prop:energymon}. We prove \Cref{thm:nosolitononS2} in the elliptic case by noticing that the orbits of the flow are periodic and hence the monotonicity formula forces the soliton to be trivial.  

It was shown in \cite[Prop. 4.2]{BrandingKroencke} that any staionary point of \cref{eq:rfeqn} on $\Sph$ except the round metric is linearly unstable. As a result, the authors expected that one cannot expect convergence of the flow on $\Sph$ in general. We confirm this expectation by giving an explicit family of connections with metric torsion such that the global solution to \cref{eq:rfeqn} exists for all time but fail to converge. More precisely, we prove the following theorem which confirms the expectation in \cite[Remark 4.3]{BrandingKroencke}.

\begin{theorem}\label{thm:noconvergence}
Let $g(t)$ be any solution of \cref{eq:rfeqn} starting wiht the round metric $g_{\text{round}}$ on $\Sph$ with the normalization $r=2$ and 
\begin{align*}
\alpha=d\left(\frac12\log(1+\varepsilon x_3)\right), \qquad 0<\varepsilon<1,
\end{align*}
where $x_3$ is the third slot of the Euclidean coordinate function. Then $g(t)$ cannot converge in $C^\infty$ as $t\to\infty$. In fact, it cannot converge in $C^2$.
\end{theorem}
The idea to prove this theorem is to prove that if a solution $g(t)$ of the above system converges then it will contradict the Kazdan-Warner identity \cref{eq:kazdan-warner} from \cite{KazdanWarner}, \cite{BE}. In spite of the above result, we give a sufficient criterion for the convergence on $\Sph$ in \Cref{subsec:coclosed}.

%\begin{proposition}\label{prop:coclosed}
%Let $g_0$ be any smooth metric on $\Sph$ and let $\alpha\in\Omega^1(\Sph)$ be a smooth one-form satisfying
%\begin{align}\label{eq:coclosedconv1}
%\diver_{g_0}\alpha=0.
%\end{align}
%Then the adapted Ricci flow \cref{eq:rfeqn} is the ordinary Ricci flow on $\Sph$, hence it converges to a metric of constant curvature by the results in \cite{Hamilton}, \cite{chow-2sphere} and \cite{chen-lu-tian}.
%\end{proposition}

\begin{proposition}[Convergence for co-closed torsion]
\label{prop:coclosed}
Let $g_0$ be any smooth metric on $\Sph$, and let $\alpha\in\Omega^1(\Sph)$ be a smooth one-form satisfying
\begin{align}\label{eq:coclosedconv1}
		\diver_{g_0}\alpha=0.
\end{align}
Let $g(t)$ be the solution to \cref{eq:rfeqn} with fixed torsion one-form $\alpha$. Then $g(t)$ exists for all $t\geq 0$  and converges smoothly to a metric $g_\infty$ satisfying $R^T_{g_\infty}\equiv r$.
\end{proposition}

Even though the previous proposition gives a sufficient condition for convergence of the flow on $\Sph$ and we show that there are infinitely many metric connections with torsion which stisfy the hypothesis of the proposition, the main underlying idea for it to work is that the flow \cref{eq:rfeqn} does not see the co-closed part of the torsion $\alpha$. Thus, we seek a better criterion which guarantees convergence. This is the content of \Cref{sec:conditionalconv}. In \Cref{subsec:antipodal-convergence} we show that one such condition is antipodal symmetry of the initial metric as well as the torsion one-form $\alpha$ on $\Sph$. Precisly, we prove the following theorem. 

\begin{theorem}[Convergence under antipodal symmetry]
	\label{thm:antipodal-convergence}
Let $g_{0}$ be a smooth metric on $\Sph$ and let $\alpha\in\Omega^{1}(\Sph)$ be a smooth one-form satisfying
\begin{align}
\iota^{*}g_{0}=g_{0} \qquad\text{and}\qquad \iota^{*}\alpha=\alpha,
		\label{eq:antipodal-data}
\end{align}
where $\iota$ is the antipodal map on $\Sph$. Then the normalized adapted Ricci flow with initial metric $g_{0}$ and fixed torsion one-form $\alpha$ converges smoothly to a metric $g_{\infty}$ with constant scalar curvature with torsion
\begin{align}
		R^{T}_{g_{\infty}}\equiv r.
		\label{eq:antipodal-limit}
\end{align}
\end{theorem}

In fact, this result and the condition was obtained when trying to understand the convergence of the flow on $\bRP^2$. As explained in \cite{BrandingKroencke}, the situation of convergence of the flow on $\bRP^2$ is different than that on $\Sph$ because the former admits metrics which are linearly stable stationary points of the flow unlike, for instance, any metric other than the round metric on $\Sph$. It was mentioned in \cite{BrandingKroencke} that it is not clear whether one can expect convergence of the flow \cref{eq:rfeqn} on $\bRP^2$ in general. We answer this question completely by proving that the flow converges to a metric with constant scalar curvature with torsion. 

\begin{theorem}[Convergence on $\bRP^2$]
	\label{thm:RP2-convergence}
Let $g_{0}$ be any smooth metric on $\bRP^2$ and let $\alpha\in\Omega^{1}(\bRP^2)$ be any smooth one-form.  Then the solution $g(t)$ of the normalized adapted Ricci flow \cref{eq:rfeqn} with initial metric $g_{0}$ and fixed torsion one-form $\alpha$ converges smoothly as $t\rightarrow\infty$ to a metric $g_{\infty}$ satisfying
\begin{align}
		R^{T}_{g_{\infty}}\equiv r.
		\label{eq:RP2-limit}
\end{align}
\end{theorem}
Before proving this theorem we explain why we get unconditional convergence on $\bRP^2$ unlike that on $\Sph$. The difficulty on $\Sph$ is related to the noncompact conformal group of $\Sph$. The idea is that $r_{\Sph}\Vol_{\Sph}=8\pi$ and $r_{\bRP^2}\Vol_{\bRP^2}=4\pi$ which forces an energy-type functional $\mathcal F$ in \cref{eq:mean-field-energy} from \cite{Cas15} to be coercive on the space of mean-zero functions on $\bRP^2$ but not on $\Sph$. As a by-product of our proof of \Cref{thm:RP2-convergence}, we recover the convergence result of \cite[Thm. 1.1]{BrandingKroencke} for surfaces with $\chi(M)\leq 0$. 

\medskip

As mentioned above, one cannot expect convergence of the flow, in general, on $\Sph$. However, we seek to find other conditions which will guarantee convergecne on $\Sph$. We observe in \Cref{subsec:conditional-convergence} that the flow \cref{eq:rfeqn} is the gradient flow of the energy in \cref{eq:energy} and hence we can hope to apply the general methods of convergence of a flow near its stationary points following Simon \cite{Sim83}. To this end, we prove a {\L}ojasiewicz--Simon gradieent inequality for the flow in \Cref{prop:LS-adapted-flow}. Using this, we prove our final conditional convergence for the adapted Ricci flow on $\Sph$. We prove in \Cref{thm:conditional-convergence} that if any solution $g(t)$ of the flow is closed in a $W^{2,2}$ neighbourhood of a stationary point $\bar{g}$ (see \cref{eq:trapping-assumption} for the precise conditions) then the flow $g(t)$ converges smoothly to a stationary point $g_{\infty}$ of the flow \cref{eq:rfeqn}. 

\medskip

We note that \cref{eq:rfeqn}, when written in terms of a conformal factor, belongs to the class of mean-field-type flows studied by Castéras. Consequently, the convergence theorem on $\bRP^2$ can also be deduced from the subcritical convergence theory in \cite{Cas15}, and the antipodally symmetric result on $\Sph$ is related to the equivariant theory in \cite{Cas13}. We nevertheless give a direct proof in the present geometric setting. This proof has several purposes. First, it expresses the compactness mechanism directly in terms of the torsion scalar curvature and the Liouville-type energy of the adapted Ricci flow. Second, it applies without change to nonorientable surfaces, which is important for treating $\bRP^2$ intrinsically rather than only through its orientable double cover. Third, it establishes smooth convergence of the evolving metrics by combining coercivity, elliptic and parabolic regularity, and a Łojasiewicz–Simon gradient inequality. Finally, the resulting finite-length and trapping estimates are also used in our conditional convergence theorem near stationary metrics on $\Sph$.

\ack{We thank Volker Branding and Klaus Kröncke for their interest and comments on an earlier draft of the paper. The author acknowledges support by the Deutsche Forschungsgemeinschaft (DFG, German Research Foundation) under Germany’s Excellence Strategy - EXC 2121 "Quantum Universe" - 390833306.}

\section{Solitons on $\Sph$ and a nonconvergence result}
	
\subsection{Triviality of solitons on $\Sph$}
	
In this section, we prove \Cref{thm:nosolitononS2} that any soliton of the normalized adapted Ricci flow on $\Sph$ is trivial, thus confirming the conjecture of \cite{BrandingKroencke}.

We first define the following Liouville type energy for solutions of the normalized adapted Ricci flow on a closed Riemann surface. Let $g_0$ be metric on a closed Riemann surface $M$ and let $g(t)=e^{u(t)}g_0$ solve the normalized adapted Ricci flow \cref{eq:rfeqn} with fixed torsion one-form $\alpha$. Define the Liouville type energy by
\begin{align}\label{eq:energy}
\mathcal E(u) = \int_{M} \left( \frac12|\nabla u|_{g_0}^2 + (R_{g_0}+2\diver_{g_0}\alpha )u - re^u \right)d\mu_{g_0}.
\end{align}
We show in \Cref{subsec:conditional-convergence} how the flow \cref{eq:rfeqn} is a gradient flow of the energy. We have the following monotonicity formula.	

\begin{proposition}\label{prop:energymon}
Along the adapted Ricci flow $g(t)$, the energy $\cE$ satisfies
\begin{align}\label{eq:energymon}
\frac{d}{d t}\mathcal E(u(t)) = - \int_{M} \left(R^T_{g(t)}-r\right)^2 d\mu_{g(t)},
\end{align}
with equality at time $t$ if and only if $R^T\equiv r$ at that time. Thus, the critical points of
$\mathcal E$ are exactly the metrics with constant curvature with torsion.
\end{proposition}
	
\begin{proof}
For an arbitrary smooth variation $u+s\eta$, integration by parts gives
		
\begin{align*}
\left.\frac{d}{d s}\right|_{s=0}\mathcal E(u+s\eta) 
&=  \int_{M} \left( \langle\nabla u,\nabla\eta\rangle_{g_0} + (R_{g_0}+2\diver_{g_0}\alpha)\eta - re^u\eta \right)d\mu_{g_0} \\
&=\int_{M} \left( -\Delta_{g_0}u+R_{g_0}+2\diver_{g_0}\alpha-re^u \right)\eta\,d\mu_{g_0}.
\end{align*}
By \cref{eq:scalartorsionconf}, we have
\begin{align*}
R_{g_0}+2\diver_{g_0}\alpha-\Delta_{g_0}u=e^uR^T_g \quad \implies \quad 	-\Delta_{g_0}u+R_{g_0}+2\diver_{g_0}\alpha-re^u = e^u(R^T_g-r).
\end{align*}
Since along the flow we have $\pt_t u=r-R^T_g$, we get
\begin{align*}
	\frac{d}{d t}\mathcal E(u(t))= \int_{M} e^u(R^T_g-r)\pt_t u\,d\mu_{g_0} = -\int_{M}
			e^u(R^T_g-r)^2\,d\mu_{g_0}=-\int_{M} (R^T_g-r)^2\,d\mu_g,
		\end{align*}
which proves \cref{eq:energymon}.
\end{proof}

We also recall the Kazdan--Warner identity from \cite{KazdanWarner}, \cite{BE}. Let $g$ be a smooth Riemannian metric on $\Sph$, and let $X$ be a conformal vector field for $g$. Then
\begin{align}\label{eq:kazdan-warner}
\int_{\Sph}X(R_g)\,d\mu_g=0.
\end{align}
Note that If $g=e^u g_{\mathrm{round}}$, where $g_{\mathrm{round}}$ is the unit round metric, the conformal vector fields are exactly the real parts of the holomorphic vector fields on the Riemann sphere. In particular, for the gradient conformal fields generated by the coordinate functions $x_1,x_2,x_3$, identity \cref{eq:kazdan-warner} becomes
\begin{align}\label{eq:kwapp}
\int_{\Sph} \left\langle \nabla_{g_{\mathrm{round}}}R_g, \nabla_{g_{\mathrm{round}}}x_j \right\rangle_{g_{\mathrm{round}}} e^u\,d\mu_{g_{\mathrm{round}}}=0.
\end{align}

If $(g, \alpha, X)$ is a non-trivial soliton on $\Sph$ then from \cref{eq:soleqn} we know that $X$ is a conformal vector field. Since every conformal structure on $\mathbb S^2$ is the standard one, the flow of $X$ is conjugate to a one-parameter subgroup of the group $\operatorname{PSL}(2,\mathbb C)$ of Möbius transformations.

We shall use the following standard classification of such vector fields. More about M\"obius transformations and properties about elliptic and nonelliptic conformal vector fields can be found in \cite[Chapter 4]{beardon}.

\begin{lemma}\label{lem:mobius-alternative}
Let $X\not\equiv0$ be a conformal vector field on $\mathbb S^2$. Then exactly one of the following occurs.
\begin{enumerate}
\item [(i)] $X$ is of elliptic type. The flow of $X$, $\varphi_t$ is periodic which means that there exists $T>0$ such that $\varphi_T=\operatorname{id}$.
\item [(ii)] $X$ is of nonelliptic type. This means that there is a zero $p$ of $X$ such that for every point outside the zero set, either $\varphi_t(q)\to p$ as $t\to+\infty$, or the analogous statement holds after replacing $X$ by $-X$.
	\end{enumerate}
\end{lemma}

\begin{remark}
The nonelliptic vector fields includes the hyperbolic, parabolic, and loxodromic one-parameter Möbius transformations.	\demo
\end{remark}

We have the following lemma which is essentially \cite[Lemma 4.4, Prop. 4.6]{BrandingKroencke}
\begin{lemma}\label{lem:invariant-potential}
Suppose that $X$ is a conformal vector field on $\Sph$ and $\cL_X\alpha=0$. By the Hodge
decomposition
\begin{equation}\label{hodge}
\alpha=dv+\omega, \qquad \diver_g\omega=0.
\end{equation}
Then $X(v)=0.$ In particular, if $X$ is nonelliptic then $v$ is constant and $\diver_g \alpha=0$.
\end{lemma}

\begin{proof}
	From $\cL_X\alpha=0$ and \cref{hodge} we obtain
	\begin{align*}
	d(Xv)+\cL_X\omega=0.
	\end{align*}	
The flow of a conformal vector field consists of conformal diffeomorphisms. Since in dimension two, co-closedness of one-forms is conformally invariant hence $\cL_X\omega$ is co-closed. It follows that $d(Xv)$ is both exact and co-closed, and therefore
\begin{align*}
\Delta_g(Xv)=0.
\end{align*}
Since $\mathbb S^2$ is closed we get that the function $Xv$ is constant. Since every conformal vector field on $\mathbb S^2$ has a zero we see that this constant is zero. Thus $X(v)=0$.

\medskip	
	
Now suppose that $X$ is nonelliptic. By \Cref{lem:mobius-alternative}, every non-fixed orbit has a limiting fixed point in one time direction. Since $v$ is constant along orbits and is continuous, we get that
\begin{align*}
v(q)=\lim_{t\to+\infty}v(\varphi_t(q))=v(p).
\end{align*}
Thus $v$ is constant on $\mathbb S^2$. Moreover,
	\begin{align*}
	\diver_g \alpha =\diver_gdv=\Delta_gv=0,
	\end{align*}
which completes the proof of the lemma.
\end{proof}

We now prove the triviality of solitons of the flow \cref{eq:rfeqn} on $\Sph$. %thus confirming the conjecture in \cite{BrandingKroencke}.

\begin{proof}[Proof of \Cref{thm:nosolitononS2}]
If the vector field $X\equiv 0$ then the result holds so lets assume $X\neq 0$. Suppose $(g, \alpha,  X)$ is a soliton on $\Sph$. If $X$ is nonelliptic then by \Cref{lem:invariant-potential} $\diver_g\alpha=0$. Hence $R^T_g=R_g$ from \cref{eq:scalar-torsion} and the soliton equation becomes
\begin{equation*}\label{ordinary-soliton}
	(r-R_g)g=\cL_Xg.
\end{equation*}
Thus $g$ is a two-dimensional Ricci soliton for the ordinary normalized Ricci flow.	Thus, using the Kazdan--Warner identity \cref{eq:kazdan-warner} we see from \cite[Prop. 5.21]{ChowKnopf} that the soliton must be trivial.

\medskip

If $X$ is elliptic then the one-parameter group generated by $X$ is periodic. Thus there
exists $T>0$ such that the flow $\varphi_t$ of $X$ satisfies
\begin{align}\label{eq:perflowmetric}
\varphi_T=\operatorname{id}_{\Sph}.
\end{align}
Define $g(t)=\varphi_t^*g$. Since $\cL_X\alpha=0$ we get
\begin{align*}
	%\label{eq:perflowform}
\varphi_t^*\alpha=\alpha.
\end{align*}
Since the scalar curvature and the divergence are diffeomorphism invariant hence $R^T_{\varphi_t^*g} = \varphi_t^*R^T_g.$ Thus,
\begin{align*}
\frac{\partial g(t)}{\partial t}
= \varphi_t^*(\cL_Xg) = \varphi_t^*\bigl((r-R^T_g)g\bigr)=
\bigl(r-R^T_{g(t)}\bigr)g(t).
\end{align*}
and as a result $g(t)$ is a solution of the normalized adapted Ricci flow.

\medskip

Choose a fixed background metric $g_0$ in the conformal class and write $g(t)=e^{u(t)}g_0.$ From \Cref{prop:energymon} we have
\begin{align*}
\frac{d}{d t}\mathcal E(u(t)) = - \int_{\Sph} (R^T_{g(t)}-r)^2\,d\mu_{g(t)}.
\end{align*}
On the other hand, \cref{eq:perflowmetric} gives $g(T)=\varphi_T^*g=g=g(0).$ Therefore
\begin{align*}
u(T)=u(0) \qquad \text{and} \qquad \mathcal E(u(T))=\mathcal E(u(0)).
\end{align*}
Thus, we have
\begin{align*}
0 = -\int_0^T\int_{\Sph} (R^T_{g(t)}-r)^2\,d\mu_{g(t)}\,d t,  \quad \implies \quad  R^T_{g(t)}\equiv r\ \text{for\ all}\ t.
\end{align*}
In particular,
\begin{align*}
R^T_g\equiv r \qquad \text{and}	 \qquad \cL_Xg=0,
\end{align*}
which is precisely the triviality of the soliton from \Cref{def:soliton}. This completes the proof of the theorem.
\end{proof}

\subsection{Non-convergence of the flow on $\Sph$}\label{subsec:nonconvergence}
In this section, we prove \Cref{thm:noconvergence} by constructing a fixed torsion one-form for which no stationary adapted metric exists on $\Sph$.
	
Let $g_{\mathrm{round}}$ be the unit round metric so that
\begin{align*}
R_{g_{\mathrm{round}}}=2, \qquad \Vol(\Sph,g_{\mathrm{round}})=4\pi, \qquad 	r=2.
\end{align*}
Let
\begin{align*}
x_3:\Sph\longrightarrow[-1,1]
\end{align*}
be the third slot of the Euclidean coordinate function. We want to use the Kazdan--Warner identity \cref{eq:kazdan-warner} and its consequence on the coordinate functions as in \cref{eq:kwapp}. To this end, choose $0<\varepsilon<1$ and define the positive function
\begin{align*}
	%\label{eq:gcurv1}
K=1+\varepsilon x_3.
\end{align*}
Set
\begin{align}\label{eq:gcurv2}
v=\frac12\log K
\end{align}
and choose the fixed torsion one-form
\begin{align}\label{eq:gcurv3}
	\alpha=d v.
\end{align}
We have the following observation.
	
\begin{proposition}\label{prop:no-stationary}
For the torsion one-form $\alpha$ in \cref{eq:gcurv3}, there is no metric conformal to $g_{\mathrm{round}}$ on $\Sph$ having constant scalar curvature with torsion equal to $2$.
\end{proposition}
	
\begin{proof}
We prove the result by contradiction. Suppose, on the contrary, that $g=e^u g_{\mathrm{round}}$ satisfies
\begin{align}\label{eq:gcurv4}
R^T_g=2.
\end{align}
Using \cref{eq:scalartorsionconf}, $R_{g_{\mathrm{round}}}=2$, and $\diver_{g_{\mathrm{round}}}\alpha=\Delta_{g_{\mathrm{round}}}v$, \cref{eq:gcurv4} gives
%\begin{align*}
%e^{-u}
%\left(
%2-\Delta_{g_{\mathrm{round}}}u
%+2\Delta_{g_{\mathrm{round}}}v
%\right)
%=2.
%\end{align*}
\begin{align}\label{eq:gcurv5}
		2-\Delta_{g_{\mathrm{round}}}(u-2v)=2e^u.
\end{align}
If we denote by $z=u-2v$, \cref{eq:gcurv5} reads as
\begin{align}\label{eq:gcurv7}
	2-\Delta_{g_{\mathrm{round}}}z = 2e^{z+2v}.
\end{align}
Consider the conformal metric $\widetilde g=e^z g_{\mathrm{round}}$ whose scalar curvature satisfies
\begin{align*}
R_{\widetilde g} = e^{-z} \left( 2-\Delta_{g_{\mathrm{round}}}z\right) \stackrel{(\ref{eq:gcurv7})}{=}	R_{\widetilde g}=2e^{2v}.
\end{align*}
Since scalar curvature is twice the Gaussian curvature in dimension two we have $K_{\widetilde g}=e^{2v}$ and by the definition of $v$ in \cref{eq:gcurv2}
\begin{align}\label{eq:gcurv8}
e^{2v}=K=1+\varepsilon x_3.
\end{align}
Thus the function $1+\varepsilon x_3$ would be the Gaussian curvature of the conformal metric $e^z g_{\mathrm{round}}$ and hence $ \nabla_{g_{\mathrm{round}}}K = \varepsilon\nabla_{g_{\mathrm{round}}}x_3$.

Using the consequence \cref{eq:kwapp} of the Kazdan--Warner identity for $j=3$ and for the metric $e^zg_{\mathrm{round}}$ we have 
\begin{align*}
\int_{\Sph} \left\langle \nabla_{g_{\mathrm{round}}}K, \nabla_{g_{\mathrm{round}}}x_3 \right\rangle_{g_{\mathrm{round}}} e^z\,d\mu_{g_{\mathrm{round}}} = \varepsilon
		\int_{\Sph}
		|\nabla_{g_{\mathrm{round}}}x_3|_{g_{\mathrm{round}}}^2
		e^z\,d\mu_{g_{\mathrm{round}}}=0.
\end{align*}
But $\nabla_{g_{\mathrm{round}}}x_3$ vanishes only at the north and south poles. Hence, as $\varepsilon >0$, we have
\begin{align*}
\varepsilon \int_{\Sph} |\nabla_{g_{\mathrm{round}}}x_3|^2	e^z\,d\mu_{g_{\mathrm{round}}}	>0,
\end{align*}
which gives us a contradiction. Therefore no solution of \cref{eq:gcurv4} exists.
\end{proof}

\medskip	

\begin{proof}[Proof of \Cref{thm:noconvergence}]
Let $g(t)$ be any solution of the flow \cref{eq:rfeqn} starting with the round metric. Long-time existence of for $g(t)$ follows \cite[Thm. 1.1]{BrandingKroencke}. Let $\alpha$ be the torsion $1$-form from \cref{eq:gcurv3}. If we look at the normalized
adapted Ricci flow starting from $g_{\mathrm{round}}$, because the flow is conformal, we have $g(t)=e^{u(t)}g_{\mathrm{round}}$ with the conformal factor satisfies
\begin{align}\label{eq:gcurv11}
\pt_t u = 2-e^{-u} \left(2-\Delta_{g_{\mathrm{round}}}u+2\Delta_{g_{\mathrm{round}}}v \right).
\end{align}
Suppose, for contradiction, that $u(t)\longrightarrow u_\infty$ in $C^2(\Sph).$ If we define the continuous nonlinear operator
\begin{align*}
\mathcal N(u) = 2-e^{-u} \left( 2-\Delta_{g_{\mathrm{round}}}u + 2\Delta_{g_{\mathrm{round}}}v\right)=\pt_tu.
\end{align*}
Since $mathcal N:C^2(\Sph)\to C^0(\Sph)$ is continuous, we have
\begin{align*}
\pt_tu(t)=	\mathcal N(u(t)) \longrightarrow \mathcal N(u_\infty) \qquad \text{in} C^0(\Sph).
\end{align*}
We claim that $ \mathcal N(u_\infty)=0.$ Suppose not and suppose we have $\mathcal N(u_\infty)(p)>0$ at some $p\in\Sph$. By continuity there exist a neighbourhood $U\ni p$, a constant $c>$ and a time $T$ such that
\begin{align*}
	\pt_t u(x,t)\geq c\quad \text{for\ every}\ x\in U\ \text{and}\  t\geq T.
\end{align*}
 In particular,
\begin{align*}
u(p,t)-u(p,T) =	\int_T^t \pt_tu(p,s)\,d s\geq c(t-T),
\end{align*}
which contradicts convergence of  $u(p,t)$. Thus, $\cN(u_{\infty})=0$ which means 
\begin{align*}
e^{-u_\infty}\left(2-\Delta_{g_{\mathrm{round}}}u_\infty + 2\Delta_{g_{\mathrm{round}}}v
		\right)=2.
\end{align*}
But this means that $e^{u_\infty}g_{\mathrm{round}}$ would be a metric of constant scalar
curvature with torsion equal to $2$. This contradicts \Cref{prop:no-stationary}. Hence the adapted Ricci flow cannot converge in $C^2$, and therefore cannot converge in $C^\infty$.
\end{proof}

\begin{remark}
Since the Kazdan-Warner identity \cref{eq:kazdan-warner} is independent of the initial metric, hence for arbitrary $g_0$ with $r=\frac{8\pi}{V}$, the same computation proves that the flow fails to converge in $C^2$ for every initial metric on $\Sph$ by choosing appropriate $\alpha$ depending on the initial metric. \demo
\end{remark}

\section{Convergence on $\Sph$ and $\bRP^2$}\label{sec:conditionalconv}
	
\subsection{A sufficient condition for convergence on $\Sph$}\label{subsec:coclosed}
	
Although a complete optimal convergence criterion on $\Sph$ is realted to the prescribed Gaussian curvature problem, there is a natural and nontrivial class of torsion one-forms for which convergence follows unconditionally.

\begin{proof}[Proof of \Cref{prop:coclosed}]
The proof essentially follows from the fact that the divergence is conformally invariant in dimension $2$. Since the flow remains in the conformal class of $g_0$ there is a smooth
function $u(t)$ such that $g(t)=e^{u(t)}g_0$. Since the co-differential is $d^*_g\alpha=-*_{g}d *_{g}\alpha$ and on one-forms in dimension two, the Hodge star is conformally invariant, whereas the Hodge star on two-forms acquires the factor $e^{-u}$, we have
\begin{align*}
\delta_{e^ug_0}\alpha = e^{-u}\delta_{g_0}\alpha,
\end{align*}
and hence
\begin{align}\label{eq:coclosedcon5}
\diver_{e^ug_0}\alpha = e^{-u}\diver_{g_0}\alpha.
\end{align}
Using \cref{eq:coclosedconv1}, \cref{eq:coclosedcon5} gives $\diver_{g(t)}\alpha=0$ for every $t\geq0$. Consequently,
\begin{align*}
R^T_{g(t)} = R_{g(t)}+2\diver_{g(t)}\alpha = R_{g(t)} \implies \frac{\partial g}{\partial t} =
(r-R_g)g,
\end{align*}
which is the ordinary normalized Ricci flow on $\Sph$. The result now follows from  classical convergence theorem for the normalized Ricci flow on a closed surface of positive Euler characteristic from \cite{Hamilton}, \cite{chow-2sphere} and \cite{chen-lu-tian} as the limiting metric $g_{\infty}$ satisfies $R_{g_\infty}\equiv r,$ and we have $\diver_{g(t)}\alpha=0$ for all $t\geq 0$ which gives
\begin{align*}
R^T_{g_\infty} = R_{g_\infty} + 2\diver_{g_\infty}\alpha = r.
\end{align*}
Thus the limiting metric has constant scalar curvature with respect to the prescribed vectorial torsion.
\end{proof}
	
\begin{remark}
The hypothesis $\diver_{g_0}\alpha=0$ does not require the torsion to vanish. On $\Sph$, since $H^1(\Sph;\mathbb R)=0$, by the Hodge decomposition theorem every co-closed one-form can be written as
\begin{align*}
\alpha=*_{g_0}d v,
\end{align*}
for some smooth function $v$. Choosing any nonconstant $v$ gives $\alpha\not\equiv 0$ while still satisfying $ \diver_{g_0}\alpha=0$. Thus \Cref{prop:coclosed} applies to an infinite-dimensional family of genuinely nonzero torsion tensors. In fact, \Cref{prop:coclosed} shows that the normalized adapted Ricci flow with fixed torsion one-form $\alpha$ converges smoothly to a metric $g_\infty$ such that
\begin{align*}
R^T_{g_\infty}\equiv \frac{8\pi}{\Vol(\Sph,g_0)}.
\end{align*}
\demo
\end{remark}

\subsection{A {\L}ojasiewicz--Simon gradient inequality}
\label{subsec:conditional-convergence}

We now want to prove other conditional convergence results for solutions on $\Sph$ which remain close to a stationary point of the adapted Ricci flow. We are also interested in analyzing the convergence of the flow to a staionary point on $\bRP^2$. Recall from \cite[Prop. 4.2]{BrandingKroencke} that a non-round stationary metric on $\Sph$ is linearly unstable. Consequently, one cannot expect every solution starting close to such a metric to converge to it or any other metric in its diffeomorphism orbit. Nevertheless, the adapted Ricci flow is the gradient flow of the Liouville-type energy introduced in \cref{eq:energy}. Since this energy is real analytic, we prove a version of {\L}ojasiewicz--Simon gradient inequality, following \cite{Sim83}, which shows that a solution which remains sufficiently close to a stationary metric cannot oscillate indefinitely. Such a solution has finite length and converges smoothly to a stationary metric.

Let $h$ be a fixed background metric on a closed surface $M$, let $\alpha\in\Omega^1(M)$ be the fixed torsion one-form and let $g(t)=e^{u(t)}h.$ Using \cref{eq:scalartorsionconf} we denote the operator 
\begin{align}\label{eq:Euler-Lagrange-operator}
	\mathcal M(u) = -\Delta_hu+R_h^T-re^u = 	e^u\bigl(R_{e^uh}^T-r\bigr),
\end{align}
and hence
\begin{align}\label{eq:conditional-gradient-flow}
	\pt_tu=-e^{-u}\mathcal M(u).
\end{align}
Recall also the Liouville-type energy from \cref{eq:energy}
\begin{equation*}
	%\label{eq:conditional-Liouville-energy}
\mathcal E(u)=\int_{M}\left( \frac12|\nabla u|_h^2+R_h^T u-re^u\right)d\mu_h.
\end{equation*}
For every $\eta\in W^{2,2}(M,h)$, its first variation is
\begin{align}
	D\mathcal E(u)[\eta] &=\int_{M} \left(\langle\nabla u,\nabla\eta\rangle_h +R_h^T\eta-re^u\eta \right)d\mu_h = \int_{M}\mathcal M(u)\eta\,d\mu_h.
	\label{eq:first-variation-conditional}
\end{align}
Thus $\mathcal M$ is the $L^2(M,h)$-gradient of $\mathcal E$.
On the other hand, \cref{eq:conditional-gradient-flow} shows that
the adapted Ricci flow is the negative gradient flow with respect to
the $u$-dependent inner product
\begin{align*}
	\langle \eta,\zeta\rangle_u = \int_{M}\eta\zeta e^u\,d\mu_h.
\end{align*}
In particular, as shown in \Cref{prop:energymon} 
\begin{align}\label{eq:conditional-energy-identity}
	\frac{d}{dt}\mathcal E(u(t)) =	-\int_{M}(\pt_tu)^2e^u\,d\mu_h	= -\int_{M}\bigl(R_{g(t)}^T-r\bigr)^2\,d\mu_{g(t)}.
\end{align}

We first establish the {\L}ojasiewicz--Simon inequality for any closed Riemann surface which will be used later.
			
\begin{proposition}[{\L}ojasiewicz--Simon inequality]\label{prop:LS-adapted-flow}
Let $(M,h)$ be a closed Riemannian surface and $\alpha\in\Omega^1(M)$ a fixed one-form. Let $\widetilde{g}=e^{\widetilde{u}}h$ be a smooth stationary metric for the normalized adapted Ricci flow, so that
				\begin{align}\label{eq:stationary-u-star}
					-\Delta_h\widetilde{u}+R_h^T=re^{\widetilde{u}}.
				\end{align}
				Then there exist constants $\sigma>0, \ C>0$ and $\theta\in\left(0,\frac12\right]$ such that
\begin{align}\label{eq:LS-adapted-flow}
\left| \mathcal E(u)-\mathcal E(\widetilde{u})\right|^{1-\theta}\leq C\|\mathcal M(u)\|_{L^2(M,h)} \qquad \text{whenever} \qquad \|u-\widetilde{u}\|_{W^{2,2}(M,h)}<\sigma.
				\end{align}
				Equivalently, after possibly decreasing $\sigma$ and enlarging $C$, we have
				\begin{align}	\label{eq:LS-curvature-form}
					\left|\mathcal E(u)-\mathcal E(\widetilde{u})\right|^{1-\theta}\leq C \left( \int_{M} \bigl(R_{e^uh}^T-r\bigr)^2\,d\mu_{e^uh}\right)^{1/2}.
				\end{align}
			\end{proposition}
			
\begin{proof}
We apply the abstract {\L}ojasiewicz--Simon gradient inequality in \cite[Thm. 3.10, Cor. 3.11]{Chi03}; see also the original work of Simon \cite{Sim83}.
				
Since $M$ has dimension two, $W^{2,2}(M,h)$ is continuously embedded in $C^0(M)$, and is a Banach algebra, it follows that
\begin{align*}
	W^{2, 2}(M,h)\longrightarrow W^{2, 2}(M,h), \qquad u\longmapsto e^u,
\end{align*}
is real analytic. Consequently,
				\begin{align*}
					\mathcal E:W^{2,2}(M,h)\longrightarrow\mathbb R \qquad \text{and} \qquad \mathcal M:W^{2,2}(M,h)\longrightarrow L^2(M,h)
				\end{align*}
				are real-analytic maps. Since by \cref{eq:stationary-u-star}, $\mathcal M(\widetilde{u})=0,$ so $\widetilde{u}$ is a critical point of $\mathcal E$. The linearization of $\mathcal M$ at $\widetilde{u}$ is
				\begin{align} \label{eq:linearized-M}
					D\mathcal M(\widetilde{u})[\eta] =-\Delta_h\eta-re^{\widetilde{u}}\eta,
				\end{align}
			and	the operator
				\begin{align*}
					D\mathcal M(\widetilde{u}):W^{2,2}(M,h)\longrightarrow L^2(M,h)
				\end{align*}
				is a second-order elliptic, self-adjoint Fredholm operator of index zero. Thus all the hypotheses of the abstract {\L}ojasiewicz--Simon inequality from \cite[Thm. 3.10]{Chi03} are satisfied, and \cref{eq:LS-adapted-flow} follows.
				
				\medskip
				
				We now prove the equivalent form \cref{eq:LS-curvature-form} of the inequality. Again by the Sobolev embedding $W^{2,2}(M,h)\hookrightarrow C^0(M)$, hence after possibly decreasing $\sigma$ if necessary, there is a constant $A>0$ such that
				\begin{align*}
					|u|\leq A \qquad \text{whenever} \qquad \|u-\widetilde{u}\|_{W^{2,2}}<\sigma.
				\end{align*}
				Using the fact that $\mathcal M(u)= e^u\bigl(R_{e^uh}^T-r\bigr)$ we obtain
				\begin{align*}
					\|\mathcal M(u)\|_{L^2(M,h)}^2=\int_{M}e^{2u}\bigl(R_{e^uh}^T-r\bigr)^2\,d\mu_h = \int_{M}e^u\bigl(R_{e^uh}^T-r\bigr)^2\,d\mu_{e^uh}.
				\end{align*}
				Since $u$ is uniformly bounded hence $e^u$ is uniformly bounded above and below in the neighbourhood described above, the last expression is uniformly equivalent to $\int_{M} \bigl(R_{e^uh}^T-r\bigr)^2\,d\mu_{e^uh}.$ This proves \cref{eq:LS-curvature-form}.
			\end{proof}

\subsection{Convergence on $\Sph$ under antipodal symmetry and on $\bRP^2$}
\label{subsec:antipodal-convergence}

Even though \Cref{prop:coclosed} gives a sufficient condition for convergence of the flow on $\Sph$, it essentially follows from the simple fact that the flow \cref{eq:rfeqn} does not see the co-closed part of $\alpha$ through the scalar curvature. We would like to strengthen the criterions for convergence of the flow. The purpose of this subsection is to give another sufficient condition for convergence of the adapted Ricci flow on $\Sph$ which allows the divergence of the torsion one-form to be nonzero.

We first make an obervation that the non-convergent example in \Cref{subsec:nonconvergence} shows that convergence cannot be guaranteed by a universal smallness assumption on $\alpha$ alone. This is because we can simply choose
\begin{align*}
\alpha_\varepsilon=\frac12\,d\log(1+\varepsilon x_3)\longrightarrow0 \quad\text{in }C^\infty \quad \text{as} \quad \varepsilon\to0.
\end{align*}
Indeed, the torsion in that example can be chosen arbitrarily small.  The difficulty is instead related to the noncompact conformal group of $\Sph$ and loss of compactness of the solutions. This suggests imposing a symmetry which rules out concentration of energy at a single point.

We prove in this section that the metric and the torsion one-form being antipodal symmetric gives us sufficient condition for convergence on $\Sph$. The reason is that antipodally invariant data on $\Sph$ descend to $\bRP^{2}$.  On this quotient the equation is better behaved and for which a convergence theory is available. This is related to the fact that $\bRP^2$ admits metrics such that first eigenvalue of the Laplacian satisfy $\lambda_1>r$ on $\bRP^{2}$, an observation already made in  \cite[Remark 4.3]{BrandingKroencke}. Besides giving a nontrivial convergence
criterion on $\Sph$, this also answers the question concerning convergence on
$\bRP^{2}$ raised in \cite[Remark 4.3]{BrandingKroencke}.

We again use a simple variant of the energy in \cref{eq:energy}. Let $h$ be a fixed background metric and let $\alpha\in\Omega^{1}(M)$ be the fixed torsion one-form.  We have
\begin{align}
R^T_{h}=R_{h}+2\diver_{h}\alpha \qquad\text{and}\qquad
	\rho=\int_{M}R^T_{h}\,d\mu_{h}=4\pi\chi(M).
	\label{eq:Q-rho}
\end{align}
If $g(t)=e^{u(t)}h$, recall from \cref{eq:scalartorsionconf} that
\begin{align}
R^{T}_{g(t)} =e^{-u(t)}\bigl(R^T_{h}-\Delta_{h}u(t)\bigr).
	\label{eq:torsion-curvature-Q}
\end{align}

\begin{lemma}\label{lem:mean-field-formulation}
Let $g(t)=e^{u(t)}h$ be a solution of the normalized adapted Ricci flow with
	fixed torsion one-form $\alpha$.  Then $u$ satisfies
\begin{align}
\frac{\partial}{\partial t}e^{u} =\Delta_{h}u-R^T_{h}
		+\rho\frac{e^{u}}{\displaystyle\int_{M}e^{u}\,d\mu_{h}}.
		\label{eq:mean-field-formulation}
\end{align}
Moreover, the functional
	\begin{align}
\mathcal{F}(u) =\frac12\int_{M}|\nabla u|_{h}^{2}\,d\mu_{h}
		+\int_{M}R^T_h u\,d\mu_{h}
		-\rho\log\left(\int_{M}e^{u}\,d\mu_{h}\right)
		\label{eq:mean-field-energy}
	\end{align}
is non-increasing along the flow and satisfies
	\begin{align}
\frac{d}{dt}\mathcal{F}(u(t)) =-\int_{M}(\pt_tu)^{2}e^{u}\,d\mu_{h}.
		\label{eq:mean-field-dissipation}
\end{align}
\end{lemma}

\begin{proof}
	Since the normalized adapted Ricci flow preserves the volume, we have
	\begin{equation}
\int_{M}e^{u(t)}\,d\mu_{h} =\Vol(M,g(t)) =\Vol(M,g(0))
		\label{eq:volume-conformal-factor}
	\end{equation}
for every $t\geq 0$. Multiplying $\pt_tu=r-e^{-u}\bigl(R^T_{h}-\Delta_{h}u\bigr)$ by $e^{u}$ gives \cref{eq:mean-field-formulation}. The monotonicity formula in \cref{eq:mean-field-dissipation} follows exactly in the same way as in \Cref{prop:energymon}.
\end{proof}

We now explain why the idea of using the functional $\cF$ and the {\L}ojasiewicz-Simon inequality can work on $\bRP^2$ but not on $\Sph$. Notice that the flow on $\Sph$, for which $\rho=8\pi$, is critical, whereas on $\bRP^{2}$ one has $\rho=4\pi<8\pi$. We explain heuristically why this distinction is important. Since $\int_M R_h^T\,d\mu_h=\rho$, as explained above the functional $\cF$ is invariant under the addition of constants. Indeed, for every $c\in \bR$ we have
\begin{align*}
\mathcal F(u+c)&=\frac12\int_M|\nabla u|_h^2\,d\mu_h
			+\int_MR_h^T(u+c)\,d\mu_h
			-\rho\log\left(\int_Me^{u+c}\,d\mu_h\right)\\
			&=
			\frac12\int_M|\nabla u|_h^2\,d\mu_h
			+\int_MR_h^T u\,d\mu_h+c\rho
			-\rho c
			-\rho\log\left(\int_Me^u\,d\mu_h\right)\\=\mathcal F(u).
\end{align*}
We may therefore replace $u$ by $u-\overline u$, where
\begin{align*}
\overline u =\frac{1}{\Vol(M,h)}\int_Mu\,d\mu_h,
\end{align*}
and hence assume that $\overline u=0$. We use the Poincar\'e inequality
\begin{align}\label{eq:Poincare-inequality}
\int_M|u-\overline u|^2\,d\mu_h \leq \frac{1}{\lambda_1(h)}\int_M|\nabla u|_h^2\,d\mu_h,
\end{align}
where $\lambda_1(h)>0$ is the first nonzero eigenvalue of $-\Delta_h$. We also use the logarithmic form of the Moser--Trudinger inequality \cite{moser2}, \cite[eq. (1.18)]{Cas15}
\begin{align}\label{eq:log-Moser-Trudinger}
\log\left(\frac{1}{\Vol(M,h)}\int_Me^{u-\overline u}\,d\mu_h\right) \leq
			\frac{1}{16\pi}
			\int_M|\nabla u|_h^2\,d\mu_h+C_h,
\end{align}
where $C_h$ depends only on the background metric $h$. With the normalization $\overline u=0$, the volume term can be absorbed into the constant, giving
\begin{align}\label{eq:log-Moser-Trudinger-zero-mean}
\log\left(\int_Me^u\,d\mu_h\right) \leq \frac{1}{16\pi}\int_M|\nabla u|_h^2\,d\mu_h+C_h'.
\end{align}
Denote by 
\begin{align*}
	%\label{eq:mean-zero-torsion-curvature}
\bigl(R_h^T\bigr)_0 = R_h^T-\frac{\rho}{\Vol(M,h)} \implies \int_M\bigl(R_h^T\bigr)_0\,d\mu_h=0,
\end{align*}
which on using $\overline u=0$ gives
\begin{align*}
\int_MR_h^T u\,d\mu_h =\int_M\bigl(R_h^T\bigr)_0u\,d\mu_h.
\end{align*}
We use the Cauchy--Schwarz and Poincar\'e inequalities to estimate
\begin{align*}
\left|\int_M\bigl(R_h^T\bigr)_0u\,d\mu_h \right| &\leq \|\bigl(R_h^T\bigr)_0\|_{L^2(M,h)} \|u\|_{L^2(M,h)}\\
&\leq	\lambda_1(h)^{-\frac12} \|\bigl(R_h^T\bigr)_0\|_{L^2(M,h)}\|\nabla u\|_{L^2(M,h)},
\end{align*}
and hence by the Young's inequality we get, for every $\varepsilon>0$,
\begin{align}\label{eq:linear-torsion-curvature-estimate}
\int_MR_h^T u\,d\mu_h \geq -\varepsilon\int_M|\nabla u|_h^2\,d\mu_h -C_\varepsilon.
\end{align}
On the other hand, since $\rho>0$,\cref{eq:log-Moser-Trudinger-zero-mean} gives
\begin{align}\label{eq:exponential-term-estimate}
-\rho\log\left(\int_Me^u\,d\mu_h\right) \geq -\frac{\rho}{16\pi} \int_M|\nabla u|_h^2\,d\mu_h-C.
\end{align}
Thus, combining \cref{eq:linear-torsion-curvature-estimate} and \cref{eq:exponential-term-estimate}, we obtain
\begin{align}\label{eq:subcritical-coercivity}
\mathcal F(u) &\geq \left( \frac12-\frac{\rho}{16\pi}-\varepsilon\right)\int_M|\nabla u|_h^2\,d\mu_h-C_\varepsilon.
\end{align}
		
If $0<\rho<8\pi$, then
\begin{align*}
\frac12-\frac{\rho}{16\pi}>0.
\end{align*}
We may therefore choose $\varepsilon>0$ sufficiently small so that the coefficient in \cref{eq:subcritical-coercivity} is positive. Consequently, $\mathcal F$ is coercive on the space of mean-zero functions. Hence we can expect convergence in the case of $\bRP^2$ (and in fact on any closed Riemann surface with $\rho < 8\pi$) but not neccessarily in the case of $\Sph$. Along the flow, $\mathcal F$ is non-increasing, and hence  $\mathcal F(u(t)) \leq\mathcal F(u(0))$. The coercivity estimate therefore gives a uniform bound for $\|\nabla u(t)\|_{L^2(M,h)}$. The Poincar\'e inequality then gives a uniform bound for $\|u(t)\|_{L^2(M,h)}$, and consequently 
\begin{align*}
		\sup_{t\geq0}\|u(t)\|_{W^{1,2}(M,h)}<\infty.
\end{align*}
In a nutshell, the decreasing energy controls the $W^{1,2}$-norm of the conformal factor.  When $M=\Sph$, the value $\rho=8\pi$ makes the preceding estimate critical and the coercivity is lost. This is the explanation for concentration along the noncompact
conformal group of $\Sph$.

\medskip

We now turn the coercivity estimate \cref{eq:subcritical-coercivity} into a proof of
\Cref{thm:RP2-convergence}. Rather than appealing to the convergence theory for mean field
type flows developed in \cite{Cas15}, we give a self-contained argument which uses only
\cref{eq:subcritical-coercivity}, the maximum principle, interior parabolic estimates, and
the {\L}ojasiewicz--Simon inequality of \Cref{prop:LS-adapted-flow}. The argument is valid on
every closed surface with $\rho<8\pi$, and in particular recovers, with a different proof,
the convergence statement of \cite[Thm. 1.1]{BrandingKroencke} for $\chi(M)\leq0$; see
\Cref{rem:all-nonpositive-euler}.

\medskip

Throughout this discussion $(M,h)$ is an arbitrary closed Riemannian surface, orientable or
not, $\alpha\in\Omega^{1}(M)$ is the fixed torsion one-form, $R^T_{h}$ and $\rho$ are as in
\cref{eq:Q-rho}, and $g(t)=e^{u(t)}h$ is the solution of \cref{eq:rfeqn} with $g(0)=h$, which
exists for all $t\geq0$ and is smooth on $M\times[0,\infty)$ by
\cite[Thm. 1.1]{BrandingKroencke}. By \cref{eq:volume-conformal-factor} the quantities
\begin{align}
	V_{0}=\int_{M}e^{u(t)}\,d\mu_{h}=\Vol(M,g(t)) \qquad\text{and}\qquad r=\frac{\rho}{V_{0}}
	\label{eq:fixed-volume-and-r}
\end{align}
are independent of $t$. We denote the right hand side of \cref{eq:mean-field-dissipation} by
\begin{align}
	\mathcal D(t) =\int_{M}(\pt_tu)^{2}e^{u}\,d\mu_{h}
	=\int_{M}\bigl(R^T_{g(t)}-r\bigr)^{2}\,d\mu_{g(t)}
	=-\frac{d}{dt}\mathcal F(u(t)).
	\label{eq:dissipation-abbreviation}
\end{align}
Since $\pt_{t}e^{u}=e^{u}\pt_tu$, \cref{eq:mean-field-formulation} may be rewritten as the
elliptic equation
\begin{align}
	-\Delta_{h}u =-R^T_{h}+re^{u}-e^{u}\pt_tu,
	\label{eq:elliptic-form-of-flow}
\end{align}
whose right-hand side has vanishing $h$-mean: indeed
$\int_{M}\bigl(re^{u}-R^T_{h}\bigr)\,d\mu_{h}=rV_{0}-\rho=0$ by \cref{eq:Q-rho} and
\cref{eq:fixed-volume-and-r}, while
$\int_{M}e^{u}u_{t}\,d\mu_{h}=\frac{d}{dt}\Vol(M,g(t))=0$. Finally, recall the operator
$\mathcal M$ from \cref{eq:Euler-Lagrange-operator}; the stationary points of \cref{eq:rfeqn}
lying in the conformal class of $h$ and having volume $V_{0}$ are exactly the solutions of
$\mathcal M(u)=0$ with $\int_{M}e^{u}\,d\mu_{h}=V_{0}$. We first prove the following \emph{a priori} bound on the evolving conformal factor $u(t)$. Before stating the result and the proof let us make a technical remark. The discussion leading to \cref{eq:exponential-term-estimate} and thereafter uses the fact that $\rho >0$ whereas all the results are stated (and indeed hold) for $\rho < 8\pi$. In fact, $\rho \leq 0$ does not cause any problem and is explained in detail in \Cref{rem:all-nonpositive-euler}.

\begin{lemma}
\label{lem:apriori-bounds}
Let $(M^2, h)$ be a closed Riemann surface. Suppose $\rho<8\pi$. Then there exist constants $\Lambda$ and $\Lambda_{p}$, $1\leq p<\infty$, depending only on $h$, $\alpha$ and $u(0)$, such that for every $t\geq0$
	\begin{align}
		\|u(t)\|_{W^{1,2}(M,h)}\leq\Lambda \qquad\text{and}\qquad
		\bigl\|e^{u(t)}\bigr\|_{L^{p}(M,h)}\leq\Lambda_{p}.
		\label{eq:uniform-W12-and-Lp}
	\end{align}
	Moreover $\mathcal F$ is bounded from below along the flow, and
	\begin{align}
		\int_{0}^{\infty}\mathcal D(t)\,dt =\mathcal F(u(0))-\mathcal F_{\infty}<\infty,
		\qquad
		\mathcal F_{\infty}=\lim_{t\rightarrow\infty}\mathcal F(u(t)).
		\label{eq:total-dissipation-finite}
	\end{align}
\end{lemma}

\begin{proof}
	As observed before \cref{eq:Poincare-inequality}, both $\mathcal F$ and
	$\|\nabla u\|_{L^{2}(M,h)}$ are invariant under adding constants to $u$. Hence the estimate
	\cref{eq:subcritical-coercivity}, derived there under the normalization $\overline u=0$, is
	valid for every $u\in W^{1,2}(M,h)$, that is
	\begin{align*}
		\mathcal F(u)\geq\kappa\int_{M}|\nabla u|_{h}^{2}\,d\mu_{h}-C_{\varepsilon},
		\qquad
		\kappa =\frac12-\frac{\rho}{16\pi}-\varepsilon>0,
	\end{align*}
	where $\varepsilon>0$ is fixed once and for all so small that $\kappa>0$ which is possible
	precisely because $\rho<8\pi$. Since $\mathcal F(u(t))\leq\mathcal F(u(0))$ by
	\cref{eq:mean-field-dissipation}, we obtain
	\begin{align}
		\|\nabla u(t)\|_{L^{2}(M,h)}^{2}\leq\Lambda_{1}^{2}
		=\kappa^{-1}\bigl(\mathcal F(u(0))+C_{\varepsilon}\bigr).
		\label{eq:uniform-gradient-bound}
	\end{align}
	We next bound the mean value $\overline u(t)$. We have the bbound  $\int_{M}e^{u-\overline u}\,d\mu_{h}\geq\Vol(M,h)$ by Jensen's inequality, while the Moser--Trudinger inequality
	\cref{eq:log-Moser-Trudinger} and \cref{eq:uniform-gradient-bound} give
	\begin{align*}
		\int_{M}e^{u-\overline u}\,d\mu_{h}
		\leq\Vol(M,h)\exp\left(\frac{\Lambda_{1}^{2}}{16\pi}+C_{h}\right), 
	\end{align*}
	and hence we have
	\begin{align*}
	\Vol(M, h)\leq \int_M e^{u-\overline u}\, d\mu_h \leq \leq\Vol(M,h)\exp\left(\frac{\Lambda_{1}^{2}}{16\pi}+C_{h}\right).
	\end{align*}
	Since $e^{\overline u}\int_{M}e^{u-\overline u}\,d\mu_{h}=V_{0}$ by
	\cref{eq:fixed-volume-and-r}, this yields
	\begin{align}
		\frac{V_{0}}{\Vol(M,h)}\exp\left(-\frac{\Lambda_{1}^{2}}{16\pi}-C_{h}\right)
		\leq e^{\overline u(t)}\leq\frac{V_{0}}{\Vol(M,h)},
		\label{eq:mean-value-bound}
	\end{align}
	and hence $|\overline u(t)|\leq C$. Together with \cref{eq:uniform-gradient-bound} and the
	Poincar\'e inequality \cref{eq:Poincare-inequality}, this proves the first bound in
	\cref{eq:uniform-W12-and-Lp}. Applying \cref{eq:log-Moser-Trudinger} to $pu$ in place of $u$
	gives
	\begin{align*}
		\log\left(\frac{1}{\Vol(M,h)}\int_{M}e^{p(u-\overline u)}\,d\mu_{h}\right)
		\leq\frac{p^{2}}{16\pi}\|\nabla u\|_{L^{2}(M,h)}^{2}+C_{h}
		\leq\frac{p^{2}\Lambda_{1}^{2}}{16\pi}+C_{h},
	\end{align*}
	which, combined with \cref{eq:mean-value-bound}, gives the second bound in
	\cref{eq:uniform-W12-and-Lp}. Finally, coercivity of the functional $\cF$ gives
	$\mathcal F(u(t))\geq-C_{\varepsilon}$, and $\mathcal F(u(t))$ is non-increasing by
	\cref{eq:mean-field-dissipation}; hence it converges to some
	$\mathcal F_{\infty}\geq-C_{\varepsilon}$, and integrating
	\cref{eq:mean-field-dissipation} in time gives \cref{eq:total-dissipation-finite}.
\end{proof}

The next lemma is the standard interior regularity theory for \cref{eq:rfeqn}. Note that the only assumption we make is that of a $C^{0}$-bound on $u$.

\begin{lemma}
	\label{lem:interior-regularity}
	Let $A>0$ and let $g(t)=e^{u(t)}h$ be a solution of \cref{eq:rfeqn} with $|u|\leq A$ on
	$M\times[a,b]$, where $b\geq a+\frac12$. Then for every integer $k\geq0$ there is a constant
	$\Gamma_{k}=\Gamma_{k}(A,k,h,\alpha)$, independent of $a$ and $b$, such that
	\begin{align}
		\sup_{t\in[a+\frac12,\,b]}\|u(t)\|_{C^{k}(M)}\leq\Gamma_{k}.
		\label{eq:interior-Ck-bound}
	\end{align}
\end{lemma}

\begin{proof}
	By \cref{eq:torsion-curvature-Q} the conformal factor satisfies
	\begin{align}
		\pt_tu=e^{-u}\Delta_{h}u-e^{-u}R^T_{h}+r.
		\label{eq:quasilinear-parabolic}
	\end{align}
Since $e^{-A}\leq e^{-u}\leq e^{A}$ by $|u|\leq A$, \cref{eq:quasilinear-parabolic} is uniformly parabolic with bounded reaction terms on the right-hand side, the ellipticity constants depending only on $A$. At this stage the coefficient $e^{-u}$ is merely bounded and measurable, so we first apply the interior H\"older estimate of Krylov--Safonov type on the cylinders $M\times[t-\frac12,t]$, which gives a uniform $C^{\beta,\frac{\beta}{2}}$-bound for some $\beta\in(0,1)$, see \cite[Chapter VII]{Lie96}. The coefficient $e^{-u}$ is then uniformly H\"older continuous in the parabolic metric, so the interior parabolic Schauder estimates apply. Differentiating \cref{eq:quasilinear-parabolic} in space and iterating, one obtains \cref{eq:interior-Ck-bound} by the standard bootstrapping argument, see \cite[Chapters IV and XII]{Lie96} or \cite{friedman-book}. All estimates are interior in time and the cylinders have fixed size, so the constants do not depend on $a$ or $b$.
\end{proof}

The following elementary comparison estimate controls the flow on the short time intervals
on which \Cref{lem:interior-regularity} gives no information. It asserts that the
$C^{0}$-distance to a stationary metric is controlled, on any fixed time interval, by its
value at the initial time of that interval.

\begin{lemma}[$C^{0}$-stability near a stationary metric]
	\label{lem:C0-stability}
Let $\widetilde u\in C^{\infty}(M)$ satisfy $\mathcal M(\widetilde u)=0$, let $g(t)=e^{u(t)}h$ be a solution of \cref{eq:rfeqn}, and set $w=u-\widetilde u$. Then
\begin{align}
\pt_t w=e^{-u}\Delta_{h}w+r\bigl(1-e^{-w}\bigr).
		\label{eq:w-flow-equation}
\end{align}
	Define
	\begin{align*}
		\Phi(b,s)=\log\bigl(1+(e^{b}-1)e^{rs}\bigr).
	%	\label{eq:comparison-function}
	\end{align*}
	If $t_{0}\geq0$ and $T>t_{0}$ are such that $\Phi\bigl(\sup_{M}w(t_{0}),\,s\bigr)$ and
	$\Phi\bigl(\inf_{M}w(t_{0}),\,s\bigr)$ are both defined for every $s\in[0,T-t_{0}]$, that
	is, both arguments of the logarithm remain positive, then
	\begin{align}
		\Phi\bigl(\inf_{M}w(t_{0}),\,t-t_{0}\bigr)\leq w(x,t)
		\leq\Phi\bigl(\sup_{M}w(t_{0}),\,t-t_{0}\bigr)
		\qquad\text{for every }(x,t)\in M\times[t_{0},T].
		\label{eq:C0-comparison}
	\end{align}
	Consequently, for every $\varepsilon>0$ and every $\tau>0$ there is
	$a=a(\varepsilon,\tau,r)>0$ such that
	\begin{align}
		\|u(t_{0})-\widetilde u\|_{C^{0}(M)}\leq a
		\quad\Longrightarrow\quad
		\|u(t)-\widetilde u\|_{C^{0}(M)}\leq\varepsilon
		\quad\text{for every }t\in[t_{0},t_{0}+\tau].
		\label{eq:C0-stability-implication}
	\end{align}
\end{lemma}

\begin{proof}
	Since $\mathcal M(\widetilde u)=0$ we have
	$R^T_{h}-\Delta_{h}\widetilde u=re^{\widetilde u}$, and therefore, by
	\cref{eq:quasilinear-parabolic},
	\begin{align*}
	\pt_tw= \pt_tu
		=r-e^{-u}\bigl(R^T_{h}-\Delta_{h}\widetilde u-\Delta_{h}w\bigr)
		=r-re^{\widetilde u-u}+e^{-u}\Delta_{h}w
		=e^{-u}\Delta_{h}w+r\bigl(1-e^{-w}\bigr),
	\end{align*}
	which is \cref{eq:w-flow-equation}. We want to apply maximum principle for functions \cite[Chapter 4]{ChowKnopf}. Let $\varphi$ solve the ordinary differential equation
	$\varphi'=r\bigl(1-e^{-\varphi}\bigr)$ with $\varphi(t_{0})=\sup_{M}w(\cdot,t_{0})=b_{+}$.
	Substituting $y=e^{-\varphi}$ turns this into $y'=ry(y-1)$, whose solution with
	$y(t_{0})=e^{-b_{+}}$ is $y(t_{0}+s)=\bigl(1+(e^{b_{+}}-1)e^{rs}\bigr)^{-1}$. Hence
	$\varphi(t_{0}+s)=\Phi(b_{+},s)$, which by hypothesis is defined on $[0,T-t_{0}]$. Since
	$\varphi$ is constant in space, the function $z=w-\varphi$ satisfies
	\begin{align*}
	\pt_tz=e^{-u}\Delta_{h}z+r\bigl(e^{-\varphi}-e^{-w}\bigr)
		=e^{-u}\Delta_{h}z+b(x,t)\,z,
	\end{align*}
	where, by the mean value theorem, $b=re^{-\xi}$ for some $\xi$ lying between $w$ and
	$\varphi$. The solution $u$ is smooth on $M\times[0,\infty)$, so $w$ is bounded on the
	compact set $M\times[t_{0},T]$, and $\varphi$ is continuous on $[t_{0},T]$, hence $b$ is
	bounded there. Since $z(\cdot,t_{0})\leq0$ (as $z(\cdot, t_0)=w(\cdot, t_0)-\sup_{M}w(\cdot, t_0)$), the maximum principle for linear parabolic
	equations with bounded zeroth-order coefficient gives $z\leq0$ on $M\times[t_{0},T]$, which
	is the upper bound in \cref{eq:C0-comparison}. The lower bound is obtained in the same way,
	comparing with the solution of the same ordinary differential equation with initial value
	$\inf_{M}w(\cdot,t_{0})$.
	
	For the last assertion, note first that
	$\pt_{b}\Phi(b,s)=e^{b}e^{rs}\bigl(1+(e^{b}-1)e^{rs}\bigr)^{-1}>0$ wherever $\Phi$ is
	defined, so $b\longmapsto\Phi(b,s)$ is increasing. Notice that $e^{rs}\leq \max\{1,e^{r\tau}\}$ for every $s\in[0,\tau]$, so choose $a>0$ so small that
	\begin{align}
		\bigl(e^{a}-1\bigr)\cdot \max\{1,e^{r\tau}\}\leq1-e^{-\varepsilon}.
		\label{eq:choice-of-a}
	\end{align}
	Assume $\|w(t_{0})\|_{C^{0}(M)}\leq a$. For the lower comparison we use
	$1-e^{-a}<e^{a}-1$, which together with \cref{eq:choice-of-a} gives, for every
	$s\in[0,\tau]$,
	\begin{align*}
		1+\bigl(e^{-a}-1\bigr)e^{rs}=1-\bigl(1-e^{-a}\bigr)e^{rs}
		\geq1-\bigl(e^{a}-1\bigr)\cdot \max\{1,e^{r\tau}\}\geq e^{-\varepsilon}>0.
	\end{align*}
	In particular both comparison functions are defined on $[t_{0},t_{0}+\tau]$, so
	\cref{eq:C0-comparison} applies with $T=t_{0}+\tau$, and it yields
	\begin{align*}
		w\geq\Phi(-a,s)\geq\log e^{-\varepsilon}=-\varepsilon.
	\end{align*}
	For the upper comparison, \cref{eq:choice-of-a} and $1-e^{-\varepsilon}\leq e^{\varepsilon}-1$
	give
	\begin{align*}
		w\leq\Phi(a,s)=\log\bigl(1+(e^{a}-1)e^{rs}\bigr)
		\leq\log\bigl(1+(e^{a}-1)\cdot \max\{1,e^{r\tau}\}\bigr)
		\leq\log\bigl(2-e^{-\varepsilon}\bigr)\leq\varepsilon,
	\end{align*}
	the last inequality because $2\leq e^{\varepsilon}+e^{-\varepsilon}$. This proves
	\cref{eq:C0-stability-implication}.
\end{proof}

\begin{lemma}[Subconvergence to a stationary metric]
	\label{lem:subconvergence}
Let $(M,h)$ be a closed Riemann surface and suppose $\rho<8\pi$. Then there exist a sequence $t_{j}\rightarrow\infty$ and a function
	$\widetilde u\in C^{\infty}(M)$ with
	\begin{align}
		\mathcal M(\widetilde u)=0 \qquad\text{and}\qquad
		\int_{M}e^{\widetilde u}\,d\mu_{h}=V_{0},
		\label{eq:limit-is-stationary}
	\end{align}
	such that $u(t_{j})\rightarrow\widetilde u$ in $W^{2,2}(M,h)$ and in $C^{0}(M)$. Moreover
	$\mathcal F_{\infty}=\mathcal F(\widetilde u)$, so that
	\begin{align}
		H(t)=\mathcal F(u(t))-\mathcal F(\widetilde u)
		=\mathcal E(u(t))-\mathcal E(\widetilde u)\geq0,
		\qquad H(t)\longrightarrow0.
		\label{eq:H-nonnegative-decreasing}
	\end{align}
\end{lemma}

\begin{proof}
	By \cref{eq:total-dissipation-finite} we may choose $t_{j}\in[j,j+1]$ with
	$\mathcal D(t_{j})\leq\int_{j}^{j+1}\mathcal D(t)\,dt\rightarrow0$. H\"older's inequality
	with $\frac23=\frac16+\frac12$, together with \cref{eq:uniform-W12-and-Lp} gives
	\begin{align}
		\bigl\|e^{u}\pt_tu\bigr\|_{L^{3/2}(M,h)}
		\leq\bigl\|e^{u/2}\bigr\|_{L^{6}(M,h)}\bigl\|e^{u/2}\pt_tu\bigr\|_{L^{2}(M,h)}
		=\bigl\|e^{u}\bigr\|_{L^{3}(M,h)}^{1/2}\,\mathcal D(t)^{1/2}
		\leq\Lambda_{3}^{1/2}\,\mathcal D(t)^{1/2},
		\label{eq:eu-ut-L32}
	\end{align}
	so the right-hand side of \cref{eq:elliptic-form-of-flow} is bounded in $L^{3/2}(M,h)$ at
	$t=t_{j}$, uniformly in $j$. As that right-hand side has vanishing mean, elliptic
	$L^{p}$-estimates give $\|u(t_{j})-\overline u(t_{j})\|_{W^{2,3/2}(M,h)}\leq C$, while
	$|\overline u(t_{j})|\leq C$ by \cref{eq:mean-value-bound}. Since $2-\frac{2}{3/2}=\frac23>0$,
	the Sobolev embedding $W^{2,3/2}(M,h)\hookrightarrow C^{0,2/3}(M)$ yields a constant
	$A_{1}$ with $\|u(t_{j})\|_{C^{0}(M)}\leq A_{1}$ for every $j$.
	
	Using this $C^{0}$-bound in place of \cref{eq:eu-ut-L32} we obtain
	\begin{align*}
		\bigl\|e^{u(t_{j})}\pt_tu(t_{j})\bigr\|_{L^{2}(M,h)}^{2}
		=\int_{M}e^{u(t_{j})}\,e^{u(t_{j})}\pt_tu(t_{j})^{2}\,d\mu_{h}
		\leq e^{A_{1}}\mathcal D(t_{j})\longrightarrow0,
	\end{align*}
	whence $\|\Delta_{h}u(t_{j})\|_{L^{2}(M,h)}\leq C$ by \cref{eq:elliptic-form-of-flow}, and
	therefore $\|u(t_{j})\|_{W^{2,2}(M,h)}\leq C$. Passing to a subsequence we find
	$\widetilde u$ with $u(t_{j})\rightharpoonup\widetilde u$ weakly in $W^{2,2}(M,h)$ and
	$u(t_{j})\rightarrow\widetilde u$ in $C^{0}(M)$, the latter by Rellich-Kondrachov compactness of the embedding $W^{2,2}(M,h)\hookrightarrow C^{0}(M)$. In particular, 
	$\int_{M}e^{\widetilde u}\,d\mu_{h}=\lim_{j}\int_{M}e^{u(t_{j})}\,d\mu_{h}=V_{0}$, and by
	\cref{eq:elliptic-form-of-flow}
	\begin{align*}
		\Delta_{h}u(t_{j})=R^T_{h}-re^{u(t_{j})}+e^{u(t_{j})}u_{t}(t_{j})
		\longrightarrow R^T_{h}-re^{\widetilde u}
		\qquad\text{in }L^{2}(M,h).
	\end{align*}
	Since $\Delta_{h}u(t_{j})\rightharpoonup\Delta_{h}\widetilde u$ weakly in $L^{2}(M,h)$, the
	limit satisfies $-\Delta_{h}\widetilde u+R^T_{h}=re^{\widetilde u}$, that is
	$\mathcal M(\widetilde u)=0$, and elliptic regularity together with bootstrapping gives $\widetilde u\in C^{\infty}(M)$. The elliptic estimate
	\begin{align*}
		\|u(t_{j})-\widetilde u\|_{W^{2,2}(M,h)}
		\leq C\left(\bigl\|\Delta_{h}\bigl(u(t_{j})-\widetilde u\bigr)\bigr\|_{L^{2}(M,h)}
		+\|u(t_{j})-\widetilde u\|_{L^{2}(M,h)}\right)\longrightarrow0
	\end{align*}
	upgrades the convergence to strong convergence in $W^{2,2}(M,h)$. Finally $\mathcal F$ is
	continuous on $W^{1,2}(M,h)$ by \cref{eq:log-Moser-Trudinger}, so
	$\mathcal F(u(t_{j}))\rightarrow\mathcal F(\widetilde u)$, and monotonicity gives
	$\mathcal F_{\infty}=\mathcal F(\widetilde u)$ and $H\geq0$. The identity in
	\cref{eq:H-nonnegative-decreasing} holds because, by \Cref{rem:F-versus-E}, one has $\mathcal F-\mathcal E=\rho(1-\log V_{0})$ on the set of all $u$ with
	$\int_{M}e^{u}\,d\mu_{h}=V_{0}$, and by \cref{eq:limit-is-stationary} this set contains
	$\widetilde u$ as well as the whole flow.
\end{proof}

The next two lemmas isolate the two mechanisms which convert the {\L}ojasiewicz--Simon
inequality into convergence. Both are used again in \Cref{thm:conditional-convergence}.

\begin{lemma}[{\L}ojasiewicz--Simon length estimate]
	\label{lem:LS-length}
	Let $\widetilde u\in C^{\infty}(M)$ satisfy $\mathcal M(\widetilde u)=0$ and let $\sigma$, $C$
	and $\theta$ be the constants of \Cref{prop:LS-adapted-flow}. Let $g(t)=e^{u(t)}h$ be a
	solution of \cref{eq:rfeqn}, let $t_{1}<t_{2}\leq\infty$ and let $A>0$ be such that, for
	every $t\in[t_{1},t_{2})$,
	\begin{align}
		\|u(t)-\widetilde u\|_{W^{2,2}(M,h)}\leq\sigma,
		\qquad |u|\leq A,
		\qquad H(t)=\mathcal E(u(t))-\mathcal E(\widetilde u)\geq0 .
		\label{eq:LS-length-hypotheses}
	\end{align}
	Then
	\begin{align}
		\int_{t_{1}}^{t_{2}}\|\pt_{t}u(t)\|_{L^{2}(M,h)}\,dt
		\leq\frac{Ce^{A/2}}{\theta}\,H(t_{1})^{\theta}.
		\label{eq:LS-length-estimate}
	\end{align}
\end{lemma}

\begin{proof}
	By \cref{eq:conditional-energy-identity} and \cref{eq:dissipation-abbreviation} we have
	$H'(t)=-\mathcal D(t)$, and by \cref{eq:LS-curvature-form}, which is available because
	$\|u(t)-\widetilde u\|_{W^{2,2}(M,h)}\leq\sigma$, we have
	$H(t)^{1-\theta}\leq C\mathcal D(t)^{1/2}$. Hence, whenever $H(t)>0$,
	\begin{align*}
		-\frac{d}{dt}H(t)^{\theta}
		=\theta\,\frac{\mathcal D(t)}{H(t)^{1-\theta}}
		\geq\frac{\theta}{C}\,\mathcal D(t)^{\frac12}.
	\end{align*}
	If $H(t_{0})=0$ for some $t_{0}\in[t_{1},t_{2})$, then, $H$ being non-negative and
	non-increasing implies that $H\equiv0$ and hence $\mathcal D\equiv0$ and $\pt_{t}u\equiv0$ on $[t_{0},t_{2})$; in this case we integrate the above inequality on $[t_{1},t_{0})$ only.
	In either case integration gives
	\begin{align*}
		\int_{t_{1}}^{t_{2}}\mathcal D(t)^{\frac12}\,dt\leq\frac{C}{\theta}H(t_{1})^{\theta}.
	\end{align*}
	Since $|u|\leq A$ implies
	$\|\pt_{t}u\|_{L^{2}(M,h)}^{2}\leq e^{A}\int_{M}(\pt_{t}u)^{2}e^{u}\,d\mu_{h}
	=e^{A}\mathcal D$, the estimate \cref{eq:LS-length-estimate} follows.
\end{proof}

\begin{lemma}[Smooth convergence from $L^{2}$-convergence]
	\label{lem:L2-to-smooth}
	Let $g(t)=e^{u(t)}h$ be a solution of \cref{eq:rfeqn}, let $u_{\infty}\in C^{\infty}(M)$ and
	let $t_{1}\geq0$ be such that
	\begin{align*}
		\sup_{t\geq t_{1}}\|u(t)\|_{C^{k}(M)}<\infty \quad\text{for every }k\geq0,
		\qquad\text{and}\qquad
		\|u(t)-u_{\infty}\|_{L^{2}(M,h)}\longrightarrow0 .
	\end{align*}
	Then $u(t)\longrightarrow u_{\infty}$ in $C^{\infty}(M)$.
\end{lemma}

\begin{proof}
	Fix an integer $k\geq0$ and choose integers $m>s>k+1$. Interpolation between $L^{2}$ and
	$W^{m,2}$ gives
	\begin{align*}
		\|u(t)-u_{\infty}\|_{W^{s,2}(M,h)}
		\leq C\|u(t)-u_{\infty}\|_{L^{2}(M,h)}^{1-\frac sm}
		\|u(t)-u_{\infty}\|_{W^{m,2}(M,h)}^{\frac sm}\longrightarrow0,
	\end{align*}
	the second factor being bounded by hypothesis. Since $\dim M=2$ and $s>k+1$, the Sobolev
	embedding gives $W^{s,2}(M,h)\hookrightarrow C^{k}(M)$, whence
	$\|u(t)-u_{\infty}\|_{C^{k}(M)}\to0$. As $k$ was arbitrary, the claim follows.
\end{proof}

We can now prove the convergence theorem. 
%As mentioned before, the \(W^{2,2}\)-convergence asserted in this theorem follows from \cite[Theorem~0.2(i)]{Cas15}, since equation~\eqref{eq:mean-field-formulation} is precisely the flow considered there with \(Q=R_h^T\). We give a self-contained proof because it exhibits directly the geometric estimates needed for the adapted Ricci flow and yields smooth convergence by parabolic regularity. 
Recall from the {\L}ojasiewicz-Simon gradient inequality  \Cref{prop:LS-adapted-flow} that, $\widetilde u$ being a smooth stationary point, there are constants $\sigma>0$, $C>0$ and
$\theta\in\left(0,\frac12\right]$ such that, by \cref{eq:LS-curvature-form} and
\cref{eq:dissipation-abbreviation},
\begin{align}
	\bigl|\mathcal E(u)-\mathcal E(\widetilde u)\bigr|^{1-\theta}
	\leq C\,\mathcal D^{\frac 12}
	\qquad\text{whenever}\qquad
	\|u-\widetilde u\|_{W^{2,2}(M,h)}<\sigma.
	\label{eq:LS-along-trajectory}
\end{align}

\begin{theorem}[Convergence for $\rho<8\pi$]
	\label{thm:subcritical-convergence}
	Let $(M,h)$ be a closed Riemannian surface and let $\alpha\in\Omega^{1}(M)$ be a smooth
	one-form. If $\rho=4\pi\chi(M)<8\pi$, then the solution $g(t)$ of the normalized adapted
	Ricci flow \cref{eq:rfeqn} with initial metric $h$ and fixed torsion one-form $\alpha$
	converges smoothly as $t\rightarrow\infty$ to a metric $g_{\infty}$ satisfying
	$R^{T}_{g_{\infty}}\equiv r$.
\end{theorem}

\begin{proof}
	Let $\widetilde u$, $t_{j}$ and $H$ be as in \Cref{lem:subconvergence}, and let $\sigma$,
	$C$ and $\theta$ be as in \cref{eq:LS-along-trajectory}. Put $A=\|\widetilde u\|_{C^{0}(M)}+1$, let $\Gamma_{k}=\Gamma_{k}(A)$ be the constants of
	\Cref{lem:interior-regularity}, let $c_{S}$ be the norm of the embedding
	$W^{2,2}(M,h)\hookrightarrow C^{0}(M)$, let $a>0$ be the constant provided by
	\cref{eq:C0-stability-implication} for $\varepsilon=\frac12$ and $\tau=1$, and set
	$\sigma'=\min\bigl\{\frac{\sigma}{2},\ \frac{1}{2c_{S}}\bigr\}$. By the Gagliardo--Nirenberg interpolation
	inequality on the closed surface $M$ there is a constant $C_{I}$ with
	$\|v\|_{W^{2,2}}\leq C_{I}\|v\|_{L^{2}}^{1/2}\|v\|_{W^{4,2}}^{1/2}$, and since
	$\|v\|_{W^{4,2}(M,h)}\leq C\|v\|_{C^{4}(M)}$ there is
	$\Theta=\Theta\bigl(\Gamma_{4},\widetilde u\bigr)$ such that
	\begin{align}
		\|u(t)-\widetilde u\|_{W^{2,2}(M,h)}
		\leq\Theta\,\|u(t)-\widetilde u\|_{L^{2}(M,h)}^{1/2}
		\qquad\text{whenever }\|u(t)\|_{C^{4}(M)}\leq\Gamma_{4}.
		\label{eq:interpolation-W22}
	\end{align}

\medskip
	
Fix $j$ so large that
	$\|u(t_{j})-\widetilde u\|_{C^{0}(M)}\leq a$, which is possible by
	\Cref{lem:subconvergence}. Then \cref{eq:C0-stability-implication} gives
	$\|u(t)-\widetilde u\|_{C^{0}(M)}\leq\frac12$ for every $t\in[t_{j},t_{j}+1]$, so that
	$|u|\leq A$ there, and \Cref{lem:interior-regularity} gives
	$\|u(t)\|_{C^{k}(M)}\leq\Gamma_{k}$ for every $t\in[t_{j}+\frac12,t_{j}+1]$. Moreover
	$e^{u}\geq e^{-A}$ on this cylinder implies
	$\|u_{t}\|_{L^{2}(M,h)}\leq e^{A/2}\mathcal D^{1/2}$, so the Cauchy--Schwarz inequality and
	\cref{eq:total-dissipation-finite} give
	\begin{align}
		\bigl\|u(t_{j}+\tfrac12)-\widetilde u\bigr\|_{L^{2}(M,h)}
		\leq\|u(t_{j})-\widetilde u\|_{L^{2}(M,h)}
		+e^{A/2}\left(\int_{t_{j}}^{t_{j}+1}\mathcal D(t)\,dt\right)^{1/2}
		\xrightarrow{\ j\rightarrow\infty\ }0.
		\label{eq:shift-by-half-a-unit}
	\end{align}
	Since $\|u(t_{j}+\frac12)\|_{C^{4}(M)}\leq\Gamma_{4}$, \cref{eq:interpolation-W22} shows
	that $\|u(t_{j}+\frac12)-\widetilde u\|_{W^{2,2}(M,h)}\rightarrow0$ as
	$j\rightarrow\infty$ and we enlarge $j$ so that in addition we have
	$\|u(t_{j}+\frac12)-\widetilde u\|_{W^{2,2}(M,h)}<\sigma'$. Thus, we moved from $t_j$ to $t_j+\frac 12$.
	
	Set $s_{j}=t_{j}+\frac12$ and let
	\begin{align*}
		T^{*}=\sup\left\{T\geq s_{j}\ :\
		\|u(t)-\widetilde u\|_{W^{2,2}(M,h)}\leq\sigma'
		\text{ for every }t\in[s_{j},T]\right\}\in(s_{j},\infty],
	\end{align*}
	and assume, for contradiction, that $T^{*}<\infty$. We will show that $T^*$ is not the supremum in this case. For $t\in[s_{j},T^{*}]$ we have
	$\|u(t)-\widetilde u\|_{C^{0}(M)}\leq c_{S}\sigma'\leq1$, and for $t\in[t_{j},s_{j}]$ we
	have $\|u(t)-\widetilde u\|_{C^{0}(M)}\leq\frac12$ as shown above. Hence $|u|\leq A$ on
	$M\times[t_{j},T^{*}]$, and \Cref{lem:interior-regularity} gives
	\begin{align}
		\|u(t)\|_{C^{k}(M)}\leq\Gamma_{k}
		\qquad\text{for every }t\in[s_{j},T^{*}]\text{ and every }k\geq0.
		\label{eq:Ck-bound-on-trapping-interval}
	\end{align}
On $[s_{j},T^{*}]$ we have $\|u(t)-\widetilde u\|_{W^{2,2}(M,h)}\leq\sigma'\leq\sigma$ and
		$|u|\leq A$, and $H\geq0$ by \cref{eq:H-nonnegative-decreasing}. Hence \Cref{lem:LS-length}
		applies on $[s_{j},T^{*}]$ and gives
		\begin{align}
			\int_{s_{j}}^{T^{*}}\|\pt_{t}u(t)\|_{L^{2}(M,h)}\,dt
			\leq\frac{Ce^{A/2}}{\theta}H(s_{j})^{\theta}=L_{j},
			\label{eq:finite-length-trapping}
		\end{align}
		and $L_{j}\rightarrow0$ as $j\rightarrow\infty$ by \cref{eq:H-nonnegative-decreasing}.

%	
%{\textbf{	Since $\|u(t)-\widetilde u\|_{W^{2,2}(M,h)}\leq\sigma'\leq\sigma$ on $[s_{j},T^{*}]$, the
%	{\L}ojasiewicz--Simon inequality \cref{eq:LS-along-trajectory} is available there. We have $H'(t)=-D(t)$ by \cref{eq:dissipation-abbreviation}, and $H(t)^{1-\theta}\leq CD(t)^{\frac 12}$ by \cref{eq:H-nonnegative-decreasing} and \cref{eq:LS-along-trajectory}. Hence, whenever
%	$H(t)>0$,
%	\begin{align*}
%		-\frac{d}{dt}H(t)^{\theta}
%		=\theta\,\frac{D(t)}{H(t)^{1-\theta}}\geq\frac{\theta}{C}\,D(t)^{\frac 12}.
%	\end{align*}
%	If $H$ vanishes at some finite time, then $\pt_u$ vanishes identically thereafter and the
%	following estimate is trivially true so we assume that $H$ doesn't vanish. Integrating from $s_{j}$ to $T^{*}$ therefore gives, in all cases,
%	\begin{align}
%		\int_{s_{j}}^{T^{*}}\|\pt_tu(t)\|_{L^{2}(M,h)}\,dt
%		\leq e^{A/2}\int_{s_{j}}^{T^{*}}D(t)^{\frac 12}\,dt
%		\leq\frac{Ce^{A/2}}{\theta}\,H(s_{j})^{\theta}=L_{j},
%		\label{eq:finite-length-trapping}
%	\end{align}
%	and $L_{j}\rightarrow0$ as $j\rightarrow\infty$ by \cref{eq:H-nonnegative-decreasing}.}}
	Consequently, for every $t\in[s_{j},T^{*}]$,
	\begin{align*}
		\|u(t)-\widetilde u\|_{L^{2}(M,h)}
		\leq\|u(s_{j})-\widetilde u\|_{L^{2}(M,h)}+L_{j}=\delta_{j},
	\end{align*}
	and $\delta_{j}\rightarrow0$ by \cref{eq:shift-by-half-a-unit}. Combining this with
	\cref{eq:Ck-bound-on-trapping-interval} and \cref{eq:interpolation-W22} we obtain
	\begin{align*}
		\|u(t)-\widetilde u\|_{W^{2,2}(M,h)}\leq\Theta\,\delta_{j}^{1/2}
		\qquad\text{for every }t\in[s_{j},T^{*}].
	\end{align*}
	Enlarging $j$ once more so that $\Theta\delta_{j}^{1/2}<\frac{\sigma'}{2}$, we conclude that
	$\|u(T^{*})-\widetilde u\|_{W^{2,2}(M,h)}<\frac{\sigma'}{2}$. Since $t\longmapsto u(t)$ is
	continuous with values in $W^{2,2}(M,h)$, the defining inequality of $T^{*}$ therefore
	persists on a slightly larger interval, contradicting the maximality of $T^{*}$. Hence
	$T^{*}=\infty$.
	
Thus, the estimates \cref{eq:Ck-bound-on-trapping-interval} and \cref{eq:finite-length-trapping} hold with $T^{*}=\infty$. In particular $\pt_t u\in L^{1}\bigl([s_{j},\infty);L^{2}(M,h)\bigr)$, so
	$u(t)$ is Cauchy in $L^{2}(M,h)$ and converges there to some $u_{\infty}$. Since
	$u(t_{j})\rightarrow\widetilde u$ in $L^{2}(M,h)$, uniqueness of the limit gives
	$u_{\infty}=\widetilde u$.

%	{\textbf{Now fix an integer $k\geq0$ and choose integers $m>s>k+1$.
%	Interpolation between $L^{2}$ and $W^{m,2}$, together with
%	\cref{eq:Ck-bound-on-trapping-interval}, gives
%	\begin{align*}
%		\|u(t)-\widetilde u\|_{W^{s,2}(M,h)}
%		\leq C\|u(t)-\widetilde u\|_{L^{2}(M,h)}^{1-\frac{s}{m}}
%		\|u(t)-\widetilde u\|_{W^{m,2}(M,h)}^{\frac{s}{m}}\longrightarrow0,
%	\end{align*}
%	and $W^{s,2}(M,h)\hookrightarrow C^{k}(M)$ because $s>k+1$ and $\dim M=2$. As $k$ was
%	arbitrary, $u(t)\rightarrow\widetilde u$ in $C^{\infty}(M)$.}}
By \cref{eq:Ck-bound-on-trapping-interval} and \Cref{lem:L2-to-smooth},
		$u(t)\rightarrow\widetilde u$ in $C^{\infty}(M)$. Therefore $g(t)=e^{u(t)}h$
	converges smoothly to $g_{\infty}=e^{\widetilde u}h$, and $\mathcal M(\widetilde u)=0$
	means, by \cref{eq:Euler-Lagrange-operator}, precisely that $R^{T}_{g_{\infty}}\equiv r$.
\end{proof}

We can now easily prove the convergence of the flow to a stationary metric on $\bRP^2$.

\begin{proof}[Proof of \Cref{thm:RP2-convergence}]
	Take $h=g_{0}$ and write $g(t)=e^{u(t)}h$, so that $u(0)=0$. Since $\chi(\bRP^{2})=1$, we
	have $\rho=4\pi<8\pi$, and \Cref{thm:subcritical-convergence} applies and gives
	\cref{eq:RP2-limit}.
\end{proof}

\begin{remark}
	\label{rem:all-nonpositive-euler}
	The hypothesis $\rho<8\pi$ is equivalent to $\chi(M)\leq1$, that is, to $M$ being any closed surface other than $\Sph$. For $\rho\leq0$ the coercivity estimate
	\cref{eq:subcritical-coercivity} is immediate: with the normalization $\overline u=0$,
	Jensen's inequality gives $\int_{M}e^{u}\,d\mu_{h}\geq\Vol(M,h)$, whence
	$-\rho\log\bigl(\int_{M}e^{u}\,d\mu_{h}\bigr)\geq-\rho\log\Vol(M,h)$, and the remaining terms are estimated exactly as in \cref{eq:linear-torsion-curvature-estimate}. Consequently \Cref{thm:subcritical-convergence} contains, in particular, the convergence statement of \cite[Thm. 1.1]{BrandingKroencke} for $\chi(M)\leq0$. The threshold $\rho=8\pi$ is attained exactly on $\Sph$, where the coercivity is lost, and by \Cref{thm:noconvergence} the conclusion genuinely fails there.
	\demo
\end{remark}

\begin{remark}
	\label{rem:nonorientable-and-Casteras}
	No step of the above argument uses an orientation of $M$: the Moser--Trudinger inequality \cref{eq:log-Moser-Trudinger}, the Poincar\'e inequality \cref{eq:Poincare-inequality}, the elliptic and parabolic estimates, and the {\L}ojasiewicz--Simon inequality of \Cref{prop:LS-adapted-flow} are all valid on an arbitrary closed Riemannian surface, with $d\mu_{h}$ interpreted as the Riemannian density and $\diver_{h}$ defined intrinsically.
	\demo
\end{remark}

We can now lift the preceding result to $\Sph$.  Let
\begin{align*}
\iota:\Sph\longrightarrow \Sph,\qquad \iota(x)=-x,
%x	\label{eq:antipodal-map}
\end{align*}
be the antipodal map.

\begin{proof}[Proof of \Cref{thm:antipodal-convergence}]
Let $g(t)$ be the solution of the normalized adapted Ricci flow \cref{eq:rfeqn}. We already know from \cite[Thm. 1.1]{BrandingKroencke} that $g(t)$ exists for all $t\geq 0$. By the naturality of scalar curvature and divergence, $\iota^{*}g(t)$ is also a solution of the same equation with the same fixed torsion one-form $\alpha$ with the initial value $\iota^{*}g(0)=\iota^{*}g_{0}=g_{0}$. Uniqueness of the solutions of the adapted flow therefore gives
\begin{align*}
\iota^{*}g(t)=g(t) \qquad\text{for every }t\geq0.
	%	\label{eq:symmetry-preserved}
\end{align*}
Thus both $g(t)$ and $\alpha$ descend under the double covering $\pi:\Sph\longrightarrow \bRP^2.$

Denote the descended metric and one-form by $\widehat{g}(t)$ and $\widehat{\alpha}$, respectively.  Since the covering map $\pi$ is a local isometry, $\widehat{g}(t)$ satisfies the normalized adapted Ricci flow on $\bRP^2$ with fixed torsion one-form $\widehat{\alpha}$.  The normalizing constants agree because both the Euler characteristic and the volume are divided by two under the quotient
\begin{align*}
\frac{4\pi\chi(\Sph)}{\Vol(\Sph,g(t))}=\frac{4\pi\chi(\bRP^2)}{\Vol(\bRP^2,\widehat{g}(t))}.
\end{align*}
By \Cref{thm:RP2-convergence}, $\widehat{g}(t)$ converges smoothly to a metric $\widehat{g}_{\infty}$ satisfying $R^{T}_{\widehat{g}_{\infty}}=r.$ Pulling this convergence back by $\pi$ gives smooth convergence of $g(t)$ to $g_{\infty}=\pi^{*}\widehat{g}_{\infty}$ and proves \cref{eq:antipodal-limit}.
\end{proof}

\begin{remark}
The condition $\iota^{*}\alpha=\alpha$ in \Cref{thm:antipodal-convergence} is a natural geometric condition which ensures that the connection with torsion descends to $\bRP^2$.  If one is interested only in the evolution of the metric, it is enough to assume
\begin{align*}
\iota^{*}g_{0}=g_{0} \qquad\text{and}\qquad \bigl(\diver_{g_{0}}\alpha\bigr)\circ\iota
=\diver_{g_{0}}\alpha,
\end{align*}
because as explained before, the metric equation depends on $\alpha$ only through its divergence. The proof by descending the full connection, however, uses the stronger and more geometric hypothesis \cref{eq:antipodal-data}. \demo
\end{remark}

We finish this section by explaining the relation with the prescribed Gaussian curvature problem and also give explicit examples to which \Cref{thm:antipodal-convergence} applies but \Cref{prop:coclosed} does not.  Let $g_{\mathrm{round}}$ be the unit round metric, and write the Hodge decomposition of the torsion one-form as
\begin{align}
	\alpha=dv+\omega,
	\qquad \operatorname{div}_{g_{\mathrm{round}}}\omega=0.
	\label{eq:hodge-antipodal}
\end{align}
If $g=e^{u}g_{\mathrm{round}}$ and $z=u-2v$, then
\begin{align*}
	R^{T}_{g}
	=e^{-2v}R_{e^{z}g_{\mathrm{round}}}.
%	\label{eq:curvature-prescription-reduction}
\end{align*}
Consequently, under the normalization $r=2$, the stationary equation $R^{T}_{g}=2$ is equivalent to
\begin{align*}
K_{e^{z}g_{\mathrm{round}}}=e^{2v},
%	\label{eq:prescribed-Gauss-antipodal}
\end{align*}
$K$ being the Gauss curvature. If $\iota^{*}\alpha=\alpha$, then $v$ can be chosen antipodally symmetric. Indeed, pullback by $\iota$ preserves both the exact and the co-closed summands in \cref{eq:hodge-antipodal}. The uniqueness of the Hodge decomposition therefore gives $d(v\circ\iota-v)=0$.  Applying $\iota$ twice shows that the resulting constant is zero.  Hence $e^{2v}$ is a positive antipodally symmetric function. We can then apply Moser's theorem \cite{moser} or \cite[Thm. 3]{chang-yang-moser} which guarantees that such a function is the Gaussian curvature of an antipodally symmetric conformal metric. In fact, \Cref{thm:antipodal-convergence} is a dynamical strengthening of this fact: it shows that the adapted Ricci flow converges to such a metric from every antipodally symmetric initial metric.

\medskip

Finally, for a concrete example, let
\begin{align*}
v=a(3x_{3}^{2}-1), \qquad \alpha=dv, \qquad a\in\mathbb{R}\setminus\{0\}.
%	\label{eq:nontrivial-antipodal-example}
\end{align*}
Then $v\circ\iota=v$ and $\iota^{*}\alpha=\alpha$, whereas
\begin{align*}
\diver_{g_{\mathrm{round}}}\alpha =\Delta_{g_{\mathrm{round}}}v\not\equiv0.
\end{align*}
Thus the adapted flow in this example is not the ordinary Ricci flow, but it
nevertheless converges smoothly by \Cref{thm:antipodal-convergence} but \Cref{prop:coclosed} is not applicable in this example. This shows that antipodal symmetry is a genuinely different convergence criterion and applies to examples not covered by \Cref{prop:coclosed}.

\begin{remark}\label{rem:F-versus-E}
	The functional $\cF$ is used in \cite{Cas15} and is equivalent to the Liouville-type energy $\mathcal E$ in \cref{eq:energy}. Indeed, if  $V(u)=\int_Me^u\,d\mu_h$ and $\rho=4\pi\chi(M)$, then the scale-invariant functional is
	
	\begin{align*}
		\mathcal F(u)=\frac12\int_M|\nabla u|^2\,d\mu_h +\int_MR_h^Tu\,d\mu_h-\rho\log V(u). 
	\end{align*}
	Since the normalized adapted Ricci flow preserves $V(u)=V_0$ and $r=\frac{\rho}{V_0}$, one has
	\begin{align*}
		\mathcal F(u)-\mathcal E(u)=\rho(1-\log V_0),
	\end{align*}
	for every $u$ with $\int_M e^u\,d\mu_h=V_0$, in particular along the flow. Hence the two functionals have the same monotonicity formula as shown above and the same critical points on the fixed-volume conformal class. The logarithmic expression is useful because it is invariant under adding constants to $u$, whereas $\mathcal E$ is more directly adapted to the fixed normalization $r$.
	\demo
\end{remark}

\begin{remark}
	\Cref{eq:mean-field-formulation} belongs precisely to the class of flows studied by Castéras in \cite{Cas13,Cas15}. Indeed, upon setting
	\begin{align*}
		Q=R_h^T, \qquad \rho=\int_M R_h^T\,d\mu_h,
	\end{align*}
	equation~\eqref{eq:mean-field-formulation} becomes
	$$
	\frac{\partial}{\partial t}e^u = \Delta_hu-Q +\rho\frac{e^u}{\int_Me^u\,d\mu_h},
	$$
	which is exactly the flow considered in \cite[eq.~(0-4)]{Cas15}.

	There is also an equivalent formulation which is useful in the presence of symmetries. Let $\psi$ be the unique mean-zero solution of
	$$
	\Delta_h\psi =R_h^T-\frac{\rho}{\operatorname{Vol}(M,h)}
	$$and set $ f=e^\psi,\  u=w+\psi.$ Since $f$ is positive and independent of time, equation~\eqref{eq:mean-field-formulation} becomes
	$$
	\frac{\partial}{\partial t}\bigl(fe^w\bigr) = \Delta_hw +\rho\left(
	\frac{fe^w}{\int_Mfe^w\,d\mu_h}-\frac{1}{\operatorname{Vol}(M,h)}\right).
	$$
	This is the equivariant flow considered in \cite{Cas13}. Consequently, the subcritical convergence result \cite[Theorem~0.2(i)]{Cas15} applies when $\rho<8\pi$, while the equivariant convergence theorem
	\cite[Theorem~1.2]{Cas13} applies when the data are invariant under an isometry group \(G\) satisfying
	$$
	\lvert G\cdot x\rvert>\frac{\rho}{8\pi} \qquad\text{for every }x\in M.
	$$
	Thus \Cref{thm:RP2-convergence} can alternatively be deduced from the subcritical convergence theorem  \cite[Theorem 0.2(i)]{Cas15}, while \Cref{thm:antipodal-convergence} follows from the equivariant convergence theorem \cite[Theorem 1.2]{Cas13}. We retain a direct proof for three reasons: it identifies explicitly the geometric coercivity mechanism for the adapted Ricci flow, treats nonorientable surfaces intrinsically, and provides the smooth compactness and finite-length estimates used later in the local convergence analysis. Accordingly, the novelty of Theorems~\ref{thm:antipodal-convergence} and \ref{thm:RP2-convergence} lies in their formulation and geometric interpretation for metric connections with vectorial torsion, rather than in a new general subcritical mean-field convergence theorem
	\demo
\end{remark}

\subsection{Another conditional convergence on $\Sph$}

We finally prove our last conditional convergence theorem.

\begin{theorem}[Conditional convergence near a stationary metric]
\label{thm:conditional-convergence}
Let $\widetilde{g}=e^{\widetilde{u}}h$ be a smooth stationary metric for the normalized adapted Ricci flow \cref{eq:rfeqn} on $\Sph$. There exists $\delta>0$ with the following property. Let $g(t)=e^{u(t)}h$ be a solution of the normalized adapted Ricci flow. Suppose that there exists $T\geq0$ such that
\begin{align}\label{eq:trapping-assumption}
\|u(t)-\widetilde{u}\|_{W^{2,2}(\Sph,h)}<\delta \qquad\text{for every }t\geq T.
\end{align}
Then there exists a smooth function $u_\infty$ such that
\begin{align*}
u(t)\longrightarrow u_\infty \qquad\text{in }\ \ C^\infty(\Sph)\ \ \text{as}\ t\to\infty.
\end{align*}
The limiting metric $g_\infty=e^{u_\infty}h$ is stationary, meaning, $R_{g_\infty}^T\equiv r$. Moreover,
\begin{align}\label{eq:same-energy-limit}
\mathcal E(u_\infty)=\mathcal E(\widetilde{u}).
\end{align}
If $\widetilde{u}$ is isolated among the solutions of \cref{eq:stationary-u-star} in a sufficiently small $W^{2,2}$-neighbourhood, then $u_\infty=\widetilde{u}$. The same conclusion holds if there exists a sequence $t_j\to\infty$ such that $u(t_j)\longrightarrow \widetilde{u}$ in $W^{2,2}(\Sph,h).$

\end{theorem}

\begin{proof}
Let $\sigma$, $C$ and $\theta$ be the constants given by \Cref{prop:LS-adapted-flow}, let
$c_{S}$ be the norm of the embedding $W^{2,2}(\Sph,h)\hookrightarrow C^{0}(\Sph)$, and set
\begin{align*}
\delta=\sigma,\qquad A=\|\widetilde u\|_{C^{0}(\Sph)}+c_{S}\delta ,
\end{align*}
with the possibility of shrinking $\delta$ in the isolated case. By \cref{eq:trapping-assumption} we have $|u|\leq A$ on
			$\Sph\times[T,\infty)$. Applying \Cref{lem:interior-regularity} on $\Sph\times[T,b]$ for
			every $b\geq T+\frac12$ we obtain constants $\Gamma_{k}=\Gamma_{k}(A)$, independent of $b$,
			with
			\begin{align}
				\|u(t)\|_{C^{k}(\Sph)}\leq\Gamma_{k}
				\qquad\text{for every }t\geq T+\tfrac12\text{ and every }k\geq0 .
				\label{eq:conditional-uniform-Ck}
			\end{align}
In particular $\left\{u(t):t\geq T+\frac12\right\}$ is precompact in $W^{2,2}(\Sph,h)$.
			
We now prove subconvergence to a stationary metric. By \cref{eq:trapping-assumption} the quantities
			$\|u(t)\|_{C^{0}(\Sph)}$ and $\|\nabla u(t)\|_{L^{2}(\Sph,h)}$ are bounded uniformly in
			$t\geq T$, so $\mathcal E(u(t))$ is bounded; being non-increasing by
			\Cref{prop:energymon}, it converges to some $\mathcal E_{\infty}\in\bR$, and integrating \cref{eq:conditional-energy-identity} gives
			\begin{align*}
				\int_{T}^{\infty}\mathcal D(t)\,dt=\mathcal E(u(T))-\mathcal E_{\infty}<\infty .
			%	\label{eq:square-integrable-gradient}
			\end{align*}
			Hence there is a sequence $t_{j}\rightarrow\infty$ with $\mathcal D(t_{j})\rightarrow0$.
			Since $|u|\leq A$,
			\begin{align*}
				\|\mathcal M(u(t_{j}))\|_{L^{2}(\Sph,h)}^{2}
				=\int_{\Sph}e^{u(t_{j})}\bigl(R^{T}_{g(t_{j})}-r\bigr)^{2}\,d\mu_{g(t_{j})}
				\leq e^{A}\mathcal D(t_{j})\longrightarrow0 ,
			\end{align*}
			where we used \cref{eq:Euler-Lagrange-operator} and
			\cref{eq:dissipation-abbreviation}. By the precompactness established above we may pass to a
			subsequence and find $u_{\infty}\in W^{2,2}(\Sph,h)$ with
			\begin{align}
				u(t_{j})\longrightarrow u_{\infty}\qquad\text{in }W^{2,2}(\Sph,h).
				\label{eq:H2-subsequence-limit}
			\end{align}
As $\mathcal M:W^{2,2}(\Sph,h)\rightarrow L^{2}(\Sph,h)$ is continuous, $\mathcal M(u_{\infty})=0$, that is $-\Delta_{h}u_{\infty}+R^{T}_{h}=re^{u_{\infty}}$ weakly, hence, elliptic regularity and bootstrapping give $u_{\infty}\in C^{\infty}(\Sph)$ and
$R^{T}_{e^{u_{\infty}}h}\equiv r$.
			
Since $\delta\leq\sigma$, the inequality \cref{eq:LS-curvature-form} holds along the trajectory for every $t\geq T$; evaluating it at $t=t_{j}$ and using \cref{eq:dissipation-abbreviation} gives
			\begin{align*}
				\bigl|\mathcal E(u(t_{j}))-\mathcal E(\widetilde u)\bigr|^{1-\theta}
				\leq C\,\mathcal D(t_{j})^{\frac12}\longrightarrow0 ,
			\end{align*}
			whence $\mathcal E_{\infty}=\mathcal E(\widetilde u)$. As $\mathcal E(u(t))$ is
			non-increasing with limit $\mathcal E(\widetilde u)$, the function
			$H(t)=\mathcal E(u(t))-\mathcal E(\widetilde u)$ satisfies $H\geq0$ and $H\rightarrow0$ on
			$[T,\infty)$. Letting $j\rightarrow\infty$ in \cref{eq:H2-subsequence-limit} and using the
			continuity of $\mathcal E$ we obtain \cref{eq:same-energy-limit}.
			
The hypotheses \cref{eq:LS-length-hypotheses} of \Cref{lem:LS-length}
			now hold on $[T,\infty)$, so
			\begin{align*}
				\int_{T}^{\infty}\|\pt_{t}u(t)\|_{L^{2}(\Sph,h)}\,dt
				\leq\frac{Ce^{A/2}}{\theta}H(T)^{\theta}<\infty .
				%\label{eq:finite-length-estimate}
			\end{align*}
			Therefore $u(t)$ is Cauchy in $L^{2}(\Sph,h)$ and converges there; by
			\cref{eq:H2-subsequence-limit} the limit is $u_{\infty}$. Combining
			\cref{eq:conditional-uniform-Ck} with \Cref{lem:L2-to-smooth} gives
			$u(t)\rightarrow u_{\infty}$ in $C^{\infty}(\Sph)$, and $g_{\infty}=e^{u_{\infty}}h$ is
			stationary as shown above.
			
			If $\widetilde u$ is isolated among the stationary points in the chosen neighbourhood, then choose an isolation radius \(\varepsilon_{\mathrm{iso}}>0\), and choose
			\begin{align*}
			 \delta<\min\{\sigma,\varepsilon_{\mathrm{iso}}\}.
			\end{align*}
Since $u(t)$ remains within $\delta$ of  $\widetilde u$ and $u(t)\to u_\infty$ in $W^{2,2}$, one has
$$ \|u_\infty-\widetilde u\|_{W^{2,2}}\leq\delta<\varepsilon_{\mathrm{iso}}. $$
As $u_\infty$ is stationary, isolation then gives $u_\infty=\widetilde u$. Likewise, if
			$u(t_{j})\rightarrow\widetilde u$ along some sequence, then the uniqueness of the
			$L^{2}$-limit implies directly that $u_{\infty}=\widetilde u$.
		\end{proof}
		
\begin{remark}
As is evident from the proof, \Cref{thm:conditional-convergence} s true for any closed Riemann surface $M$, but the most interestng case is that of $M=\Sph$ as for other cases we have an unconditional convergence theorem.  \demo
\end{remark}

\begin{corollary}[The nondegenerate case]
\label{cor:conditional-nondegenerate}
Let $\widetilde{g}=e^{\widetilde{u}}h$ be a stationary metric and suppose that
\begin{equation}\label{eq:nondegenerate-spectrum}
r\notin\operatorname{Spec}(-\Delta_{\widetilde{g}}).
\end{equation}
Then $\widetilde{u}$ is an isolated stationary point. Moreover, the exponent in \Cref{prop:LS-adapted-flow} may be chosen to be $\theta=\frac12$. Consequently, any solution satisfying the  assumption \cref{eq:trapping-assumption} converges exponentially in $C^\infty$ to $\widetilde{u}$.
\end{corollary}

\begin{proof}
Multiplying \cref{eq:linearized-M} by $e^{-\widetilde{u}}$ gives
\begin{align*}
e^{-\widetilde{u}}D\mathcal M(\widetilde{u})[\eta] = -\Delta_{\widetilde{g}}\eta-r\eta.
\end{align*}
Thus
\begin{align*}
\ker D\mathcal M(\widetilde{u}) = \left\{ \eta\in W^{2,2}(\Sph) \mid  -\Delta_{\widetilde{g}}\eta=r\eta \right\}.
\end{align*}
As a result, assumption \cref{eq:nondegenerate-spectrum} implies that $D\mathcal M(\widetilde{u})$ is invertible. The Banach space implicit function theorem shows that $\widetilde{u}$ is an isolated stationary point. For a nondegenerate critical point, the {\L}ojasiewicz--Simon exponent can be chosen to be $\theta=\frac12$.
	
%Using the same notations as in the proof of \Cref{thm:conditional-convergence}, we have
%\begin{align*}
%H(t)^{1/2}\leq CY(t) \qquad \text{and} \qquad H'(t)=-Y(t)^2\leq-\frac{1}{C^2}H(t),
%\end{align*}
%and hence
%\begin{align*}
%H(t)\leq Ce^{-ct}.
%\end{align*}
%The finite-length estimate and the interpolation argument in the proof of \Cref{thm:conditional-convergence} then give exponential convergence in every $C^k$-norm. 
Using the same notation as in the proof of \Cref{thm:conditional-convergence}, we have
$H(t)^{1/2}\leq C\mathcal D(t)^{1/2}$ and $H'(t)=-\mathcal D(t)\leq-C^{-2}H(t)$, so that
$H(t)\leq Ce^{-ct}$. \Cref{lem:LS-length} applied on $[t,\infty)$ then gives
$\|u(t)-\widetilde u\|_{L^{2}(\Sph,h)}\leq C'e^{-\frac{ct}{2}}$, and
\cref{eq:conditional-uniform-Ck} together with the interpolation in the proof of
\Cref{lem:L2-to-smooth} upgrades this to exponential convergence in every $C^{k}$-norm. We mention that this result excludes the round metric on $\Sph$ as the condition in \cref{eq:nondegenerate-spectrum} fails. 
\end{proof}

\begin{remark}
The assumption \cref{eq:trapping-assumption} in \Cref{thm:conditional-convergence} cannot in general be replaced by the assumption that $u(0)$ is merely close to $\widetilde{u}$. In fact, as shown in \cite[Prop. 4.2]{BrandingKroencke} the linearization of \cref{eq:rfeqn} at a stationary metric is
\begin{align*}
\eta_t=\Delta_{\widetilde{g}}\eta+r\eta.
\end{align*}
Thus an eigenfunction satisfying $-\Delta_{\widetilde{g}}\eta=\lambda\eta$ has linear growth rate $r-\lambda$ and hence on $\Sph$, Hersch's inequality \cite{Hersch} gives
\begin{align*}
\lambda_1(\widetilde{g})\leq \frac{8\pi}{\operatorname{Vol}(\Sph,\widetilde{g})}=r,
\end{align*}
with equality only for a round metric. Consequently, every non-round stationary metric has an unstable direction. What \Cref{thm:conditional-convergence} asserts is that trajectories of flows which do not escape along such an unstable direction nevertheless converge to a stationary metric. \demo
\end{remark}

\printbibliography

@article{Chi03,
	author  = {Chill, Ralph},
	title   = {On the {\L}ojasiewicz--Simon gradient inequality},
	journal = {Journal of Functional Analysis},
	volume  = {201},
	number  = {2},
	year    = {2003},
	pages   = {572--601},
	doi     = {10.1016/S0022-1236(02)00102-7}
}

@article{Sim83,
	author  = {Simon, Leon},
	title   = {Asymptotics for a class of nonlinear evolution
	equations, with applications to geometric problems},
	journal = {Annals of Mathematics},
	volume  = {118},
	number  = {3},
	year    = {1983},
	pages   = {525--571},
	doi     = {10.2307/2006981}
}

@book{beardon,
	author = {Beardon, Alan F.},
	title = {The geometry of discrete groups},
	fseries = {Graduate Texts in Mathematics},
	series = {Grad. Texts Math.},
	issn = {0072-5285},
	volume = {91},
	year = {1983},
	publisher = {Springer, Cham},
	language = {English},
	zbMATH = {3838373},
	Zbl = {0528.30001}
}

@article{chang-yang-moser,
	author = {Chang, Sun-Yung Alice and Yang, Paul C.},
	title = {The inequality of {Moser} and {Trudinger} and applications to conformal geometry},
	fjournal = {Communications on Pure and Applied Mathematics},
	journal = {Commun. Pure Appl. Math.},
	issn = {0010-3640},
	volume = {56},
	number = {8},
	pages = {1135--1150},
	year = {2003},
	language = {English},
	doi = {10.1002/cpa.3029},
	zbMATH = {1981619},
	Zbl = {1049.53025}
}

@misc{friedman-book,
	author = {Friedman, Avner},
	title = {Partial differential equations of parabolic type},
	year = {1964},
	language = {English},
	howpublished = {Englewood {Cliffs}, {N}.{J}.: {Prentice}-{Hall}, {Inc}. xiv, 347 p. (1964).},
	zbMATH = {3233089},
	Zbl = {0144.34903}
}

@article{moser2,
	author = {Moser, J{\"u}rgen},
	title = {A sharp form of an inequality by {Trudinger}},
	fjournal = {Indiana University Mathematics Journal},
	journal = {Indiana Univ. Math. J.},
	issn = {0022-2518},
	volume = {20},
	pages = {1077--1092},
	year = {1971},
	language = {English},
	doi = {10.1512/iumj.1971.20.20101},
	zbMATH = {3337983},
	Zbl = {0213.13001}
}

@misc{moser,
	author = {Moser, J{\"u}rgen},
	title = {On a nonlinear problem in differential geometry},
	year = {1973},
	language = {English},
	howpublished = {Dynamical {Syst}., {Proc}. {Sympos}. {Univ}. {Bahia}, {Salvador} 1971, 273-280 (1973).},
	zbMATH = {3433111},
	Zbl = {0275.53027}
}

@article{Cas15,
	author = {Cast{\'e}ras, Jean-Baptiste},
	title = {A mean field type flow. {II}: {Existence} and convergence},
	fjournal = {Pacific Journal of Mathematics},
	journal = {Pac. J. Math.},
	issn = {1945-5844},
	volume = {276},
	number = {2},
	pages = {321--345},
	year = {2015},
	language = {English},
	doi = {10.2140/pjm.2015.276.321},
	zbMATH = {6469581},
	Zbl = {1331.53097}
}

@article{BE,
	author = {Bourguignon, Jean Pierre and Ezin, Jean Pierre},
	title = {Scalar curvature functions in a conformal class of metrics and conformal transformations},
	fjournal = {Transactions of the American Mathematical Society},
	journal = {Trans. Am. Math. Soc.},
	issn = {0002-9947},
	volume = {301},
	pages = {723--736},
	year = {1987},
	language = {English},
	doi = {10.2307/2000667},
	zbMATH = {4009258},
	Zbl = {0622.53023}
}

@article{streets,
	author = {Streets, Jeffrey},
	title = {Regularity and expanding entropy for connection {Ricci} flow},
	fjournal = {Journal of Geometry and Physics},
	journal = {J. Geom. Phys.},
	issn = {0393-0440},
	volume = {58},
	number = {7},
	pages = {900--912},
	year = {2008},
	language = {English},
	doi = {10.1016/j.geomphys.2008.02.010},
	zbMATH = {5293474},
	Zbl = {1144.53326}
}

@book{brendle,
	author = {Brendle, Simon},
	title = {Ricci flow and the sphere theorem},
	fseries = {Graduate Studies in Mathematics},
	series = {Grad. Stud. Math.},
	issn = {1065-7339},
	volume = {111},
	isbn = {978-0-8218-4938-5},
	year = {2010},
	publisher = {Providence, RI: American Mathematical Society (AMS)},
	language = {English},
	zbMATH = {5673930},
	Zbl = {1196.53001}
}

@article{bfr,
	author = {Fardoun, Ali and Baird, Paul and Regbaoui, Rachid},
	title = {The evolution of the scalar curvature of a surface to a prescribed function},
	fjournal = {Annali della Scuola Normale Superiore di Pisa. Classe di Scienze. Serie V},
	journal = {Ann. Sc. Norm. Super. Pisa, Cl. Sci. (5)},
	issn = {0391-173X},
	volume = {3},
	number = {1},
	pages = {17--38},
	year = {2004},
	language = {English},
	url = {https://eudml.org/doc/84525},
	zbMATH = {2217253},
	Zbl = {1170.58306}
}

@article{chen-lu-tian,
	author = {Chen, Xiuxiong and Lu, Peng and Tian, Gang},
	title = {A note on uniformization of {Riemann} surfaces by {Ricci} flow},
	fjournal = {Proceedings of the American Mathematical Society},
	journal = {Proc. Am. Math. Soc.},
	issn = {0002-9939},
	volume = {134},
	number = {11},
	pages = {3391--3393},
	year = {2006},
	language = {English},
	doi = {10.1090/S0002-9939-06-08360-2},
	zbMATH = {5120156},
	Zbl = {1113.53042}
}

@article{chow-2sphere,
	author = {Chow, Bennett},
	title = {The {Ricci} flow on the 2-sphere},
	fjournal = {Journal of Differential Geometry},
	journal = {J. Differ. Geom.},
	issn = {0022-040X},
	volume = {33},
	number = {2},
	pages = {325--334},
	year = {1991},
	language = {English},
	doi = {10.4310/jdg/1214446319},
	zbMATH = {4214867},
	Zbl = {0734.53033}
}

@article{agricola-kraus,
	author = {Agricola, Ilka and Kraus, Margarita},
	title = {Manifolds with vectorial torsion},
	fjournal = {Differential Geometry and its Applications},
	journal = {Differ. Geom. Appl.},
	issn = {0926-2245},
	volume = {45},
	pages = {130--147},
	year = {2016},
	language = {English},
	doi = {10.1016/j.difgeo.2016.01.004},
	zbMATH = {6546696},
	Zbl = {1339.53046}
}

@article{Struwe,
	author = {Struwe, Michael},
	doi = {10.1215/S0012-7094-04-12812-X},
	fjournal = {Duke Mathematical Journal},
	issn = {0012-7094},
	journal = {Duke Math. J.},
	language = {English},
	number = {1},
	pages = {19--64},
	title = {A flow approach to {Nirenberg}'s problem},
	volume = {128},
	year = {2005},
	zbl = {1087.53034},
	zbmath = {2223075}}

@article{KazdanWarner,
	author = {Kazdan, Jerry L. and Warner, Frank W.},
	doi = {10.2307/1971012},
	fjournal = {Annals of Mathematics. Second Series},
	issn = {0003-486X},
	journal = {Ann. Math. (2)},
	language = {English},
	pages = {14--47},
	title = {Curvature functions for compact 2-manifolds},
	volume = {99},
	year = {1974},
	zbl = {0273.53034},
	zbmath = {3429742}}

@article{Hersch,
	author = {Hersch, Joseph},
	fjournal = {Comptes Rendus Hebdomadaires des S{\'e}ances de l'Acad{\'e}mie des Sciences, S{\'e}rie A},
	issn = {0366-6034},
	journal = {C. R. Acad. Sci., Paris, S{\'e}r. A},
	language = {French},
	pages = {1645--1648},
	title = {Quatre propri{\'e}t{\'e}s isop{\'e}rim{\'e}triques de membranes sph{\'e}riques homog{\`e}nes. ({Some} isoperimetric properties of spherical membranes)},
	volume = {270},
	year = {1970},
	zbl = {0224.73083},
	zbmath = {3356292}}

@misc{Hamilton,
	author = {Hamilton, Richard S.},
	howpublished = {Mathematics and general relativity, {Proc}. {AMS}-{IMS}-{SIAM} {Jt}. {Summer} {Res}. {Conf}., {Santa} {Cruz}/{Calif}. 1986, {Contemp}. {Math}. 71, 237-262 (1988).},
	language = {English},
	title = {The {Ricci} flow on surfaces},
	year = {1988},
	zbl = {0663.53031},
	zbmath = {4084495}}

@book{ChowKnopf,
	author = {Chow, Bennett and Knopf, Dan},
	fseries = {Mathematical Surveys and Monographs},
	isbn = {0-8218-3515-7},
	issn = {0076-5376},
	language = {English},
	publisher = {Providence, RI: American Mathematical Society (AMS)},
	series = {Math. Surv. Monogr.},
	title = {The {Ricci} flow: an introduction},
	volume = {110},
	year = {2004},
	zbl = {1086.53085},
	zbmath = {2121403}}

@article{ChangYang,
	author = {Chang, Sun-Yung Alice and Yang, Paul C.},
	doi = {10.1007/BF02392560},
	fjournal = {Acta Mathematica},
	issn = {0001-5962},
	journal = {Acta Math.},
	language = {English},
	pages = {215--259},
	title = {Prescribing {Gaussian} curvature on {S} 2},
	volume = {159},
	year = {1987},
	zbl = {0636.53053},
	zbmath = {4036660}}

@article{Cas13,
		author  = {Cast{\'e}ras, Jean-Baptiste},
		title   = {Equivariant mean field flow},
		journal = {Journal of Geometry and Physics},
		volume  = {74},
		year    = {2013},
		pages   = {314--327},
		doi     = {10.1016/j.geomphys.2013.08.011}
	}

@book{Lie96,
		author = {Lieberman, Gary M.},
		title = {Second order parabolic differential equations},
		isbn = {981-02-2883-X},
		year = {1996},
		publisher = {Singapore: World Scientific},
		language = {English},
		zbMATH = {1061253},
		Zbl = {0884.35001}
	}

@article{BrandingKroencke,
	author = {Branding, Volker and Kr{\"o}ncke, Klaus},
	doi = {10.1007/s12220-016-9753-4},
	fjournal = {The Journal of Geometric Analysis},
	issn = {1050-6926},
	journal = {J. Geom. Anal.},
	language = {English},
	number = {3},
	pages = {2098--2117},
	title = {The {Ricci} flow with metric torsion on closed surfaces},
	volume = {27},
	year = {2017},
	zbl = {1376.53084},
	zbmath = {6794210}}

\noindent
Fachbereich Mathematik, Universität Hamburg, Bundesstraße 55, 20146 Hamburg, Germany.\\
\href{mailto:shubham.dwivedi@uni-hamburg.de}{shubham.dwivedi@uni-hamburg.de}	
	
\end{document}